\documentclass{amsart}
\usepackage[utf8]{inputenc}
\usepackage{hyperref}
\usepackage[notref,notcite]{showkeys}
\usepackage{enumerate} 
\usepackage{upref}  
\usepackage{verbatim} 
\usepackage{color}
\usepackage{graphicx}
\usepackage{bbm}
\usepackage[textsize=tiny]{todonotes}

\newcommand*{\mailto}[1]{\href{mailto:#1}{\nolinkurl{#1}}}
\newcommand{\arxiv}[1]{\href{http://arxiv.org/abs/#1}{arXiv:#1}}
\usepackage{amssymb, amsmath, amsfonts, amsthm}
\usepackage{mathtools}

\usepackage{marginnote}

\newcommand{\RN}[1]{%
  \textup{\uppercase\expandafter{\romannumeral#1}}%
}

\usepackage{mleftright} 

\renewcommand{\left}{\mleft}
\renewcommand{\right}{\mright}

\newcommand{\epsi}{\varepsilon}

\newcommand{\X}{\ensuremath{\mathcal{X}}}
\newcommand{\Y}{\ensuremath{\mathcal{Y}}}

\newcommand{\D}{\ensuremath{\mathcal{D}}}

\newcommand{\abs}[1]{\left\vert#1\right\vert}

\newcommand{\Real}{\mathbb R}

\DeclareMathOperator{\sgn}{sgn}

\newcommand{\dee}{\,d}

\newtheorem{theorem}{Theorem}[section]

\newtheorem{lemma}[theorem]{Lemma}
\newtheorem{definition}[theorem]{Definition}
\newtheorem{proposition}[theorem]{Proposition}

\newtheorem{remark}[theorem]{Remark}

\numberwithin{equation}{section}

\allowdisplaybreaks

\begin{document}

\title[Wave breaking for the NVW equation]{Wave breaking for the nonlinear variational wave equation}

\author[S.T. Galtung]{Sondre Tesdal Galtung}
\address{Department of Mathematical Sciences\\ NTNU Norwegian University of Science and Technology\\ NO-7491 Trondheim\\ Norway}
\email{\mailto{sondre.galtung@sintef.no}}

\author[K. Grunert]{Katrin Grunert}
\address{Department of Mathematical Sciences\\ NTNU Norwegian University of Science and Technology\\ NO-7491 Trondheim\\ Norway}
\email{\mailto{katrin.grunert@ntnu.no}}
\urladdr{\url{https://www.ntnu.edu/employees/katrin.grunert}}

\thanks{We acknowledge support by the grant {\it Wave Phenomena and Stability --- a Shocking Combination (WaPheS)}  from the Research Council of Norway. }  

\begin{abstract}
Following conservative solutions of the nonlinear variational wave equation $u_{tt}-c(u)(c(u)u_x)_x=0$ along forward and backward characteristics, we identify criteria, which guarantee that wave breaking either occurs in the nearby future or occurred recently. Thereafter, we apply the established criteria to show that not every traveling wave solution is a conservative solution. Furthermore, we show that conservative solutions can locally behave like solutions to the linear wave equation and hence energy that concentrates on sets of measure zero might remain concentrated instead of spreading out immediately. 
\end{abstract}

\maketitle

\section{Introduction}

We will here study the nonlinear variational wave equation, which for an unknown scalar $u = u(t,x)$ reads
\begin{equation}\label{NVW}
	u_{tt}-c(u)(c(u)u_x)_x=0,
\end{equation}
where the real numbers $t$ and $x$ respectively represent time and position.
Equation \eqref{NVW} was derived in \cite{Sax1989} as a simplified model for the director field of a nematic liquid crystal, where $u$ represents the angle of director fields restricted to the unit circle.
The term \textit{variational} comes from the fact that \eqref{NVW} is the Euler--Lagrange equation of the Lagrangian
\begin{equation*}
	\iint \left(u_t^2 - (c(u)u_x)^2\right)\dee x \dee t
\end{equation*}
associated with the Oseen--Frank theory of liquid crystals, see \cite[Section 2]{GlaHunZhe1997} for a derivation.
On a related note, smooth solutions of \eqref{NVW} preserve the energy
\begin{equation}\label{eq:energy}
	\frac12\int_{\Real} \left(u_t^2 + (c(u)u_x)^2\right)\dee x.
\end{equation}
In \cite{HunSax1991} the authors considered weakly nonlinear, unidirectional waves satisfying \eqref{NVW}, and derived an asymptotic equation which turns out to be completely integrable.
This equation is now known as the Hunter--Saxton equation and has spawned its own vast literature, see e.g. \cite{ChrGru2024, Daf2011, GaLiWo2023, GruHo2022} and the references therein.

\subsection{Singularity formation}
The nonlinear variational wave equation \eqref{NVW} can be seen as a particular, in-between case in the one-parameter family of equations 
\begin{equation}\label{eq:1para}
	u_{tt} = c(u)^{2-\alpha}(c(u)^\alpha u_x)_x, \qquad 0 \le \alpha \le 2,
\end{equation}
see \cite{GlaHunZhe1997}.
One extreme here is the nonlinear wave equation where $\alpha=0$
\begin{equation}\label{eq:lind}
	u_{tt} - c(u)^2 u_{xx} = 0,
\end{equation}
studied in \cite{Lin1992}, which is expected to have globally smooth solutions.
The other extreme is the conservation law-like equation
\begin{equation}\label{eq:psys}
	u_{tt} - \left(c(u)^2u_x\right)_x = 0.
\end{equation}
In fact, \eqref{eq:psys} can be seen as a particular case of the so-called $p$-system \cite[Example 5.3]{holden2015front}, which reads $u_t = v_x$ and $v_t = -(p(u))_x$ for some pressure-like term $p$.
The choice $p'(u) = -c(u)^2 $ yields \eqref{eq:psys} and ensures the system is hyperbolic with eigenvalues $\pm c(u)$, and, typical for such systems, we can expect shocks and discontinuous solutions.

If we write out the spatial derivative in the one-parameter family \eqref{eq:1para} for nonzero $\alpha$, we get a quadratic term in $u_x$, and the above discussion hints at this term being an obstacle for retaining smoothness of solutions.
Indeed, it drives singularity formation such as blow-up and oscillations, cf.\ \cite{GlaHunZhe1997}.
Our emphasis is on blow-up, and to this end we consider the quantities
\begin{equation}\label{eq:RS}
	R \coloneqq u_t + c(u)u_x, \qquad S \coloneqq u_t - c(u)u_x,
\end{equation}
which are  Riemann invariants for the linear wave equation.
That is, when $c(u) \equiv c > 0$, $R$ and $S$ are constant along the backward and forward characteristics, respectively $x = -ct$ and $x = ct$.
For smooth solutions of \eqref{NVW} we may rewrite the equation as a system of equations in $R$ and $S$, namely
\begin{equation}\label{eq:RSsys}
	R_t -c(u)R_x = \frac{c'(u)}{4 c(u)} \left( R^2 - S^2 \right),  \qquad S_t + c(u) S_x = \frac{c'(u)}{4c(u)} \left( S^2 - R^2 \right).
\end{equation}
Observe that we can add the two equations in \eqref{eq:RSsys} and divide by two to recover \eqref{NVW}.
If we for simplicity replace the factors involving $c$ and $c'$ in \eqref{eq:RSsys} with $1$, we obtain what is known as the  Carleman system, for which the solutions $R$ and $S$ blow up along respectively $x = -t$ and $x = t$ for appropriate sign choices in the initial data. This blow-up mechanism is analogous to that of a Riccati equation.
The same idea is employed in \cite{GlaHunZhe1996} under additional assumptions on the sign of $c'$.
It is much harder to predict blow-up when $c'$ is allowed to change sign, as this may delay or even prevent it from happening.

Clearly, one cannot expect \eqref{NVW} to have smooth solutions in general, and so a weaker notion of solution is needed beyond their breakdown. 
Assuming $u$ smooth enough, we have the pointwise identities
\begin{equation*}
	\left(R^2 + S^2\right)_t - \left(c(u)(R^2-S^2)\right)_x = 0, \qquad \left(\frac{R^2-S^2}{c(u)}\right)_t - \left(R^2+S^2\right)_x = 0,
\end{equation*}
where we recognize $\frac14 (R^2 + S^2) = \frac12(u_t^2 + (c(u)u_x)^2)$ as the energy density from \eqref{eq:energy}.
Let us then introduce the positive Radon measures $\mu$ and $\nu$ with their respective absolutely continuous parts $\frac14 R^2 \dee x$ and $\frac14 S^2 \dee x$ with respect to the Lebesgue measure.
Based on the previous pointwise identities we then require these measures to satisfy the distributional identities
\begin{equation}\label{eq:NVW:meas}
	(\mu + \nu)_t - \left(c(u)(\mu-\nu)\right)_x = 0, \qquad \left(\frac{\mu-\nu}{c(u)}\right)_t - (\mu+\nu)_x = 0.
\end{equation}
Together with a weak formulation of \eqref{NVW}, this is how \textit{conservative solutions} of the nonlinear variational wave equation were defined in \cite{BreZhe2006}.
Here $(\mu+\nu)(t,\Real)$ is the total energy, which is conserved in time for conservative solutions.
Recalling \eqref{eq:energy}, this implies that $u_x^2$ remains integrable, which in turn means that $u(t,\cdot)$ is H\"{o}lder-continuous.
Hence, conservative solutions of \eqref{NVW} still have more regularity than we can expect from solutions of \eqref{eq:psys}.

These conservative solutions are typically studied using characteristics, see \cite{BreZhe2006,HR}, which in this case consist of two families of curves with velocities $\pm c(u)$, as opposed to the straight lines of the linear wave equation.
The aforementioned Hunter--Saxton equation and the related Camassa--Holm equation share the conservative solution concept, and this is typically studied using characteristics, cf.\ \cite{Bre2016}.
These equations model unidirectional wave propagation, and so there is only one family of characteristic curves traveling with velocity $u$, owing to the Burgers advection term $uu_x$ featured in both equations.
This makes them considerably easier to study in comparison to \eqref{NVW} with its two families of characteristics.
For instance, with these characteristics one typically changes to so-called Lagrangian coordinates where one replaces the Eulerian coordinate $x$ in a fixed reference frame with $\xi$ by following a given characteristic, i.e., $(t,x) \mapsto (t,\xi)$.
For the nonlinear variational wave equation, the change to Lagrangian coordinates involves two such labels and a coordinate change in both space and time, $(t,x) \mapsto (\xi,\eta)$, see \cite{HR}.
Furthermore, characteristics have been used to show uniqueness of conservative solutions in \cite{BCZ}.
On a similar note, we mention that \eqref{NVW} appears as an example in an application of so-called singular characteristics \cite[Section 8.4]{Melikyan}, see also \cite{KorKorMel2009}.

\subsection{Prediction of wave-breaking}
For the Hunter--Saxton and Camassa--Holm equations, blow-up happens for the slope $u_x$.
Since their energy densities contain the term $u_x^2$, $u$ will remain both bounded and H\"{o}lder-continuous while this happens, which leads to a vertical wave profile at these points.
This is called wave-breaking and leads to the breakdown of classical solutions.
An analogous mechanism drives wave-breaking for \eqref{NVW} where blow-up happens for one or both of $R$ and $S$, while $u$ remains bounded.
Indeed, one such blow-up result for \eqref{NVW} is found in \cite[Theorem 1]{GlaHunZhe1996} where the initial data takes a special form.
That is, assuming $c'(\bar{u}) \neq 0$ for some constant $\bar{u} \in \Real$, one takes 
\begin{equation*}
	u(0,x) = \bar{u} + \epsi \phi\left(\frac{x}{\epsi}\right), \qquad u_t(0,x) = -\sgn(c'(\bar{u}))c(u(0,x))u_x(0,x)
\end{equation*}
for sufficiently small $\epsi > 0$ and a nonzero, compactly supported $\phi \in C^1_{c}(0,1)$.
Moreover, $c \in C^2(\Real)$ is assumed strictly positive, and bounded from both below and above.
For such data, depending on the sign of $c'(\bar{u})$, exactly one of $R$ and $S$ will initially be identically zero.
The idea is then to show that one of $R$ and $S$ will remain of order $\epsi$, while the other will blow up along a characteristic in finite time.

Let us mention some other (non-)existence results related to the nonlinear variational wave equation \eqref{NVW}.
A local-in-time existence result for classical solutions is found in \cite{ZhaZhe2001}, together with global-in-time existence results for what they call rarefactive solutions where one assumes $c'(x)> 0$ and $R(0,x), S(0,x) \le 0$.
In the latter case, we may observe that the combination of signs prohibits $R$ and $S$ in \eqref{eq:RSsys} from blowing up along their appropriate characteristics.
Another potential problem may occur when the wave speed $c$ is not bounded away from zero; this is called finite-time degeneracy and is studied in \cite{Sug2017}.The same author \cite{Sug2022} extended the aforementioned blow-up result of \cite{GlaHunZhe1996} to \eqref{eq:1para} for $\alpha \in (0,1]$ under similar assumptions.

The aim of this paper is to give more new, more general and precise estimates on when to expect wave-breaking for \eqref{NVW}, in the spirit of earlier works on related equations, e.g., \cite{Gru2015}.
Before we summarize the content of these result, we will state our assumptions on the wave speed $c = c(u)$ in the remainder of the paper.
We assume there exists $\kappa>1$ such that 
\begin{equation}\label{cond:c}
	\frac{1}{\kappa}\leq c(x)\leq \kappa \quad \text{ for all } x\in \Real.
\end{equation}
Moreover, there exists $\lambda\geq 0$ and $\bar\lambda\geq 0$ such that 
\begin{equation}\label{cond:cder}
	\vert c'(x)\vert \leq \lambda \quad \text{ and } \quad \vert c''(x) \vert \leq \bar\lambda \quad \text{ for all } x \in \Real.
\end{equation}

Next we give a summary of our main results, that is, sufficient conditions on the initial data for us to predict wave-breaking for \eqref{NVW}, and to this end we denote by $u_0$, $R_0$, and $S_0$ the initial values of $u$, $R$, and $S$.
Moreover, let us write $E_0$ for the initial energy $(\mu_0+\nu_0)(\Real)$.
In a similar vein to previous results, we will leverage the Riccati-type blow-up mechanism, but here we shall do so in Lagrangian coordinates.
Since the blow-up of either $R$ or $S$ corresponds to wave-breaking, this enables us to give a lower bound $t_l$ and an upper bound $t_u$ for when this happens.

For instance, the quantity $R$ may blow up along backward characteristics, that is, characteristics moving with velocity $-c(u)$.
If the initial condition $R_0$ satisfies our condition for blow-up, then, depending on the sign of $c'(u_0(\bar{x})) R_0(\bar{x})$, $R$ will either have already have blown up in the past, or will blow up in the future.

To be more precise, consider a point $\bar{x} \in \Real$:
\begin{itemize}
	\item \textbf{Blow-up along backward characteristics}: If $\abs{R_0(\bar{x})}$ is large enough compared to $\lambda \kappa^2 E_0$, and additionally satisfies an inequality which involves $\abs{c'(u_0(\bar{x}))} $, then $R$ blows up along a backward characteristic.
	This blow-up happens in the future if $c'(u_0(\bar{x})) R_0(\bar{x})$ is positive, otherwise it has happened in the past.
	\item \textbf{Blow-up along forward characteristic}: If $\abs{S_0(\bar{x})}$ is large enough compared to $\lambda \kappa^2 E_0$, and additionally satisfies an inequality which involves $\abs{c'(u_0(\bar{x}))}$, then $S$ blows up along a forward characteristic.
	This blow-up happens in the future if  $c'(u_0(\bar{x})) S_0(\bar{x})$ is positive, otherwise it has happened in the past.
\end{itemize}
In both cases we are able to determine a time interval $[t_l, t_u]$ where wave-breaking happens, and we refer to Theorems \ref{thm:wb1}--\ref{thm:wb4} in Section 3 for the details.

As an application of our results, we will in Section 4 use our wave-breaking criteria to prove that a traveling wave, described in \cite{GlaHunZhe1996} as ``hut''-shaped, is not a conservative solution of the nonlinear variational wave equation.
This is done by showing that the conservative solution evolving from the initial traveling-wave profile must exhibit a form of wave-breaking which is incompatible with the evolution of the traveling wave.
As such, this is an alternative to verifying that the measure-equations \eqref{eq:NVW:meas} do not hold for this traveling wave.

We finish Section 4 with an example that highlights a major difference  between \eqref{NVW} and the Hunter--Saxton and Camassa--Holm equations, namely that concentrated energy need not immediately spread out again, but may remain concentrated over time.
This can happen when the derivative of the wave speed  $c = c(u)$ is zero.
In our example, the left- and right-moving Radon measures $\mu$ and $\nu$ associated with the concentrated energy behave locally like a solution of the linear wave equation, that is, they move with constant velocities.

\section{Background}\label{sec:back}

The prediction of wave breaking is based on a good understanding of how wave breaking can be described and identified. In the case of conservative solutions the unique weak solution to any admissible initial data can be derived via a generalized method of characteristics, see \cite{HR} and \cite{BCZ}. Thus, to predict wave breaking one needs to understand well the interplay between Eulerian and Lagrangian coordinates as well as the underlying system of partial differential equations in Lagrangian coordinates.  Since the generalized method of characteristics for conservative solutions has been discussed in detail in \cite{HR}, we here focus, on the one hand, on sketching the method of characteristics from \cite{HR} and, on the other hand, on  how wave breaking is characterized in both Eulerian and Lagrangian coordinates. 

In the case of the linear wave equation, i.e., $c(u)=c\in \Real$, every classical solution is given by d'Alemberts formula, which implies
\begin{equation*}
u(t,x)= f(x+ct)+g(x-ct).
\end{equation*}
That is, for $c > 0$, $u(t,x)$ consists of one part, $f$, traveling to the left and one part, $g$, traveling to the right with speed $c$. Therefore one can state the initial data at $t=0$ in terms of $f$ and $g$ as 
\begin{equation*}
(u_0(x), (u_{0,t}+cu_{0,x})(x), (u_{0,t}-cu_{0,x})(x))= ((f+g)(x), 2cf'(x), -2cg'(x)),
\end{equation*}
instead of using $u_0$ and $u_{0,t}$. Furthermore, every classical solution to the linear wave equation can be computed using the classical method of characteristics. This means that one introduces a change of variables $(t,x) \to (\xi, \eta)$, such that the backward and forward characteristics, given by $y(t, \xi)= \xi- ct$ and $z(t,\eta)= \eta+ct$, are mapped to vertical and horizontal lines, respectively. This is achieved by implicitly defining $(\xi, \eta) \mapsto (t(\xi, \eta), x(\xi, \eta))$ as
\begin{equation}\label{crossing:point}
x(\xi, \eta)= y(t(\xi, \eta), \xi)= z(t(\xi, \eta), \eta).
\end{equation}
Or, in other words, $(t(\xi, \eta), x(\xi, \eta))$ denotes the unique intersection point of $y(t, \xi)$ and $z(t, \eta)$. Moreover,  \eqref{crossing:point} implies that 
\begin{equation*}
x_\xi (\xi, \eta)= ct_\xi(\xi, \eta) \quad \text{ and }\quad x_\eta(\xi, \eta)=- ct_\eta(\xi, \eta),
\end{equation*}
which is the Lagrangian formulation of the forward and backward characteristics.

As a consequence, one has, in this case, that the line $\{(0,x)\mid x\in \Real\}\subset \Real^2$, for which the initial data is defined, is mapped to the curve $\bar{\mathcal{C}}=\{(\xi, \eta)\mid \xi=\eta\}$ in $\Real^2$, along which the initial data in Lagrangian coordinates is given by 3 triplets 
\begin{subequations}\label{Lagr:initdata}
\begin{align}
(t(\xi, \xi), x(\xi, \xi), U(\xi, \xi))& =(0, \xi, u_0(\xi)),\\
(t_\xi(\xi, \xi), x_\xi(\xi, \xi), U_\xi(\xi, \xi))&=(\frac1{2c}, \frac12, \frac{1}{2c}(u_{0,t}+cu_{0,x})(\xi)),\\
(t_\eta(\xi, \xi), x_\eta(\xi, \xi), U_\eta(\xi, \xi))& = (-\frac1{2c}, \frac12, -\frac{1}{2c}(u_{0,t}-cu_{0,x})(\xi) ).
\end{align}
\end{subequations}

To compute $(t, x, U)$ on all of $\Real^2$, one solves the reformulation of the linear wave equation in Lagrangian coordinates, which is given by 
\begin{equation*}
(t_{\xi, \eta}(\xi, \eta), x_{\xi, \eta}(\xi, \eta), U_{\xi, \eta}(\xi, \eta))=(0,0,0) \quad \text{ for all } (\xi, \eta)\in \Real^2, 
\end{equation*}
for the initial data given by \eqref{Lagr:initdata}. 

To finally extract the solution in Eulerian coordinates from the tuplets $((t,x,U),$ $( t_\xi, x_\xi, U_\xi), (t_\eta, x_\eta,U_\eta))$, one uses the following relations
\begin{subequations}\label{backtoEuler}
\begin{align}
u(t(\xi, \eta), x(\xi, \eta))& = U(\xi, \eta),\\
\frac{1}{c}(u_t+cu_x)(t(\xi, \eta), x(\xi, \eta))x_\xi(\xi, \eta)& = U_\xi (\xi, \eta),\\
-\frac{1}{c} (u_t-cu_x)(t(\xi, \eta), x(\xi, \eta))x_\eta(\xi, \eta)& = U_\eta(\xi, \eta).
\end{align}
\end{subequations}

In the case of the nonlinear variational wave equation, no analogue to d'Alembert's formula is known, but a generalized method of characteristics, which mimics the approach for the linear wave equation, makes it possible to follow the solution along backward and forward characteristics, which are given by 
\begin{equation}\label{def:char}
y_t(t,\xi)= -c(u(t, y(t,\xi))) \quad \text{ and }\quad z_t(t, \eta)= c(u(t, z(t, \eta))), 
\end{equation}
respectively. Thus each initial data at $t=0$ will contain triplets 
\begin{equation*}
(u_0,R_0,S_0)= (u_0, u_{0,t}+c(u_0)u_{0,x}, u_{0,t}-c(u_0)u_{0,x}).
\end{equation*}  
On the other hand, if wave breaking occurs in the future, which is not possible for the linear wave equation, either $(u_t+c(u)u_x)(t,x)$ or $(u_t-c(u)u_x)(t,x)$ become unbounded pointwise and energy may concentrate on sets of measure zero either along backward or forward characteristics, respectively. Thus the set of admissible initial data consists of tuplets $(u_0,R_0,S_0, \mu_0, \nu_0)$, where $\mu_0$ and $\nu_0$ are positive finite Radon measures, which describe the concentration of energy and split the total energy $\mu_0+\nu_0$ into two parts.

\begin{definition}[Eulerian coordinates]\label{def:eul} The set $\D$ consists of all tuplets $(u,R,S, \mu, \nu)$ such that 
\begin{equation*}
(u,R,S)=(u, u_t+c(u)u_x, u_t-c(u)u_x)\in [L^2(\Real)]^3 
\end{equation*}
and $\mu$ and $\nu$ are positive, finite Radon measures whose absolutely continuous parts $\mu_{ac}$ and $\nu_{ac}$ satisfy
\begin{equation}\label{meas:abs}
d\mu_{ac}=\frac14 R^2 dx \quad \text{ and }\quad d\nu_{ac}= \frac14 S^2 dx.
\end{equation}
\end{definition}

As for the linear wave equation, one needs to associate to any  initial data $(u_0,R_0,S_0, \mu_0,\nu_0)\in \mathcal{D}$ a corresponding set of Lagrangian coordinates $(Z, Z_\xi, Z_\eta)$, which is given along a curve 
\begin{equation*}
\bar{\mathcal{C}}=\{(\bar \X(s), \bar \Y(s))\in \Real^2 \mid s\in \Real\}.
\end{equation*}

The definition of this curve is based on the measures $(\mu_0, \nu_0)$, which contain all the information about initial energy concentration. In particular, $\bar{\mathcal{C}}$ must guarantee that when energy concentrates in single point, then the spreading out of the energy afterwards mimics in a sense the behavior of a rarefaction wave. Therefore, introduce
\begin{subequations}\label{x1x2}
\begin{align}
x_1(X)&= \sup\{ x\in \Real\mid x+ \mu_0((-\infty, x))<X\},\\
x_2(Y)& = \sup\{x\in \Real \mid x+ \nu_0((-\infty,x))<Y\},
\end{align}
\end{subequations}
which are increasing and Lipschitz continuous functions, then 
\begin{equation*}
\bar \X(s)= \sup\{X\in \Real\mid x_1(X')<x_2(2s-X') \text{ for all } X'<X\}
\end{equation*}
and $\bar\Y(s)= 2s-\bar \X(s)$. Note that both $\bar\X(s)$ and $\bar \Y(s)$ are increasing, Lipschitz continuous with Lipschitz constant at most $2$, and
\begin{equation*}
\lim_{s\to \pm\infty}\bar \X(s)=\pm\infty= \lim_{s\to\pm\infty} \bar\Y(s),
\end{equation*}
which implies that to any given $(\xi, \eta)\in \bar{\mathcal{C}}$, there exists a unique $s\in \Real$ such that $\xi+ \eta=2s$ and $(\xi, \eta)=(\bar \X(s), \bar\Y(s))$. 

Next, introduce 
\begin{equation}\label{def:barx}
\bar x(s)=x_1(\bar\X(s))=x_2(\bar \Y(s)),
\end{equation}
then 
\begin{equation}\label{def:immx}
\bar x(s)=\sup\{x\in \Real\mid 2x + (\mu_0+\nu_0)((-\infty, x))<2s\},
\end{equation}
which implies that for $\mathcal{S}=\bar x(\mathcal{B})$, where  
\begin{equation}\label{def:B}
\mathcal{B}=\{s\mid \dot{\bar  {x}}(s)=0\},
\end{equation}
one has 
\begin{equation*}
(\mu_0+\nu_0)_{\mathrm{ac}}= (\mu_0+\nu_0)\vert_{\mathcal{S}^c} \quad \text{ and } \quad  (\mu_0+\nu_0)_{\mathrm{sing}}= (\mu_0+\nu_0)\vert_{\mathcal{S}}.
\end{equation*}
Furthermore, the points at which wave breaking occurs initially in Eulerian coordinates correspond to the set $\mathcal{B}$ in Lagrangian coordinates. However, $\bar J(s)= 2s- 2\bar x(s)$ in Lagrangian coordinates corresponds to $\mu_0+\nu_0$ in Eulerian coordinates and hence does not allow to distinguish between $\mu_0$ and $\nu_0$. On the other hand, we can use \eqref{x1x2} to split $\bar J(s)$ as follows
\begin{equation*}
\bar J(s)= \bar J_1(s)+\bar J_2(s)=J_1(\bar \X(s))+J_2(\bar \Y(s))
\end{equation*}
where 
\begin{equation}\label{J1J2}
J_1(X)=X-x_1(X)\quad \text{ and }\quad J_2(Y)=Y-x_2(Y), 
\end{equation}
which will be important later. Last, but not least we introduce 
\begin{equation*}
\bar U(s)= u_0(\bar x(s)). 
\end{equation*}
With all these definitions in place we can define the initial Lagrangian coordinates along the curve $\bar{\mathcal{C}}$ by 
\begin{equation*}
Z(\xi, \eta) = (t, x, U, J)(\xi, \eta) = (0, \bar x(s), \bar U(s), \bar J(s)) \quad \text{ for }(\xi, \eta)= (\bar\X(s), \bar\Y(s))\in \bar{\mathcal{C}}.
\end{equation*}

In analogy to the linear wave equation, we also need to know $Z_\xi$ and $Z_\eta$ on the curve $\bar{\mathcal{C}}$ before we can turn our attention towards the underlying system of partial differential equations. Here it is important, that we keep the properties, which we highlighted for the linear wave equation, 
i.e., following vertical and horizonal lines in the $(\xi, \eta)$ plane corresponds to following backward and forward characteristics, respectively, since 
\begin{equation}\label{def:crp}
x(\xi, \eta)= y(t(\xi, \eta), \xi)= z(t(\xi, \eta), \eta),
\end{equation}
and 
\begin{equation}\label{rel:tx}
x_\xi(\xi, \eta)= c(U)t_\xi(\xi, \eta) \quad \text{ and } \quad x_\eta(\xi, \eta)=- c(U)t_\eta(\xi, \eta),
\end{equation}
as well as add some additional properties, which are related to \eqref{meas:abs}, namely 
\begin{equation}\label{rel:seebreak}
(c(U)U_\xi)^2(\xi, \eta)=2x_\xi J_\xi (\xi, \eta) \quad \text{ and }\quad (c(U)U_\eta)^2(\xi, \eta)= 2x_\eta J_\eta(\xi, \eta),
\end{equation}
for all $(\xi, \eta)\in \bar{\mathcal{C}}$. 

Inspired by \eqref{Lagr:initdata}, one defines $Z_\xi(\xi, \eta)$ and $Z_\eta(\xi, \eta)$ for $(\xi, \eta)= (\bar \X(s), \bar \Y(s))\in \bar{\mathcal{C}}$ as follows
\begin{subequations}\label{init:data}
\begin{align} \nonumber
Z_\xi(\xi, \eta)& = (t_\xi, x_\xi, U_\xi, J_\xi)(\xi, \eta)\\
&  = ( \frac{1}{2c(\bar U(s))}x_1'(\bar\X(s)), \frac12 x_1'(\bar\X(s)), \frac{R(\bar x(s))}{2c(\bar U(s))}x_1'(\bar\X(s)), J_1'(\bar\X(s)))\\ \nonumber
Z_\eta(\xi, \eta)& = (t_\eta, x_\eta, U_\eta, J_\eta)(\xi, \eta)\\ 
& = (-\frac1{2c(\bar U(s))} x_2'(\bar\Y(s)), \frac12 x_2'(\bar\Y(s)), - \frac{S(\bar x(s))}{2c(\bar U(s))} x_2'(\bar\Y(s)), J_2'(\bar \Y(s))). 
\end{align}
\end{subequations}
Here it is important to note that wave breaking occurs initially at all those points $(\xi, \eta)=(\bar \X(s), \bar \Y(s))\in \bar{\mathcal{C}}$ such that either $(x_\xi, U_\xi)(\xi, \eta)=(0,0)$ or $(x_\eta, U_\eta)(\xi, \eta)=(0,0)$, since $\mathcal{B}$, given by \eqref{def:B}, contains all points for which wave breaking occurs initially, and $\bar x(s)$ is related to $x_1(X)$ and $x_2(Y)$ by \eqref{def:barx}. 

 To compute $(t,x,U,J)$ on all of $\mathbb{R}^2$, in the next step, one uses a reformulation of the nonlinear variational wave equation in Lagrangian coordinates, which is given by \begin{subequations}\label{eq:Lagr}
\begin{align}
t_{\xi, \eta}& = -\frac{c'(U)}{2c(U)}(t_\xi U_\eta+ t_\eta U_\xi), \\ \label{eq:Lagr2}
x_{\xi, \eta}&= \frac{c'(U)}{2c(U)}(x_\xi U_\eta + x_\eta U_\xi),\\ \label{eq:Lagr3}
U_{\xi, \eta} & = \frac{c'(U)}{2c^3(U)} (x_\xi J_\eta+ x_\eta J_\xi)-\frac{c'(U)}{2c(U)} U_\xi U_\eta,\\
J_{\xi, \eta}& = \frac{c'(U)}{2c(U)} (J_\xi U_\eta + J_\eta U_\xi).
\end{align}
\end{subequations}
In \cite{HR} it is shown that the above system has a unique solution in a rather complicated space. For us only some properties satisfied by the solutions are of particular interest. The triplet $(Z, Z_\xi, Z_\eta)$ belongs to the following space,
\begin{equation*}
Z\in [W^{1, \infty}(\mathbb{R}^2)]^4,
\end{equation*}
 for almost every $\xi\in \Real$, 
\begin{equation*}
Z_\xi (\xi, \cdot)=(t_\xi, x_\xi, U_\xi, J_\xi)(\xi, \cdot)  \in [W^{1, \infty}(\Real)]^4, 
\end{equation*}
and for almost every $\eta\in \Real$,
\begin{equation*}
Z_\eta(\cdot, \eta)= (t_\eta, x_\eta, U_\eta, J_\eta)(\cdot, \eta) \in [W^{1, \infty}(\Real)]^4.
\end{equation*}
Moreover, observe  that by \eqref{x1x2}, \eqref{J1J2}, and \eqref{init:data} 
\begin{equation}\label{prop:G}
(2x_\xi+J_\xi)(\xi, \eta)=1 \quad \text{and} \quad (2x_\eta+J_\eta)(\xi, \eta)=1 \quad \text{for all }(\xi, \eta)\in \bar{\mathcal{C}}.
\end{equation}
Although these relations are not preserved by the system of differential equations \eqref{eq:Lagr}, a slightly weaker result holds. Namely, given $L>0$, let 
\begin{equation}\label{def:SL}
S_L=\{(\xi, \eta)\in \Real^2\mid \mathrm{dist}( (\xi, \eta), \bar{\mathcal{C}})<L\}.
\end{equation}
Then, there exist $C_1$ and $C_2$, dependent on $L$, such that we have for all $(\xi,\eta) \in S_L$ \begin{equation}\label{prop:rel}
0<C_1< (2x_\xi +J_\xi)(\xi, \eta) <C_2 \quad \text{and} \quad 0<C_1<(2x_\eta+J_\eta)(\xi,\eta)<C_2.
\end{equation} 
Furthermore, one has 
\begin{equation}\label{asymp:J1}
J(\xi, \eta) \to 0  \quad \text{ as }\quad  \xi+\eta\to -\infty,
\end{equation}
and 
\begin{equation}\label{asymp:J2}
J(\xi, \eta) \to (\mu_0+ \nu_0)(\Real) \quad \text{ as }\quad  \xi+\eta\to \infty,
\end{equation} 

Observe that \eqref{prop:rel} states that if $x_\xi(\bar \xi, \eta)$ tends to $0$ as $(\bar \xi, \eta) \to (\bar \xi, \bar \eta)$, then $J_\xi(\bar \xi, \bar \eta)$ is strictly positive, and hence by \eqref{rel:seebreak}, we have that 
\begin{equation*}
\left(c(U)\frac{U_\xi }{x_\xi}\right)^2(\bar \xi, \eta)= 2\frac{J_\xi}{x_\xi}(\bar \xi, \eta) \to \infty \quad \text{as }(\bar \xi, \eta) \to (\bar \xi, \bar \eta).
\end{equation*}
Following the same lines of reasoning, we have that if  $x_\eta(\xi, \bar \eta)$ tends to $0$ as $(\xi, \bar \eta)\to (\bar \xi, \bar \eta)$, then 
\begin{equation*}
\left(c(U)\frac{U_\eta }{x_\eta}\right)^2(\xi, \bar \eta)= 2\frac{J_\eta}{x_\eta}(\xi, \bar \eta) \to \infty \quad \text{as }(\xi, \bar \eta) \to (\bar \xi, \bar \eta).
\end{equation*}

Or in other words, if wave breaking occurs at a point $(\bar \xi, \bar \eta)$ in Lagrangian coordinates, which happens if either $x_\xi(\bar \xi, \eta)$ or $x_\eta(\xi, \bar \eta)$ tend to $0$, then 
\begin{equation}\label{rel:LagrBreak}
\left(\frac{U_\xi }{x_\xi}\right)^2(\bar \xi, \eta) \quad \text{ or } \quad \left(\frac{U_\eta }{x_\eta}\right)^2(\xi, \bar \eta) \quad \text{tend to } \infty.
\end{equation}
Furthermore, by \eqref{rel:tx}
\begin{equation}\label{signt}
t_\xi(\xi, \eta) \geq 0 \quad \text{ and }\quad t_ \eta(\xi, \eta) \leq 0,
\end{equation}
which means that the connected set 
\begin{equation*}
E_{T}=\{(\xi, \eta)\in \Real^2\mid t(\xi, \eta)=T\}
\end{equation*} 
need not be a curve in $\Real^2$. In fact, it may consist of the union of a curve with a countable number of rectangles. Furthermore, for $T>0$, the set $E_T$ will lie below $\bar{\mathcal{C}}$. Thus, to go back from Lagrangian to Eulerian coordinates at time $T$, one needs to extract a curve $\mathcal{C}_T\subset E_T$. Since this curve is in general not unique, one makes a specific choice. Namely, 
\begin{equation*}
\mathcal{C}_T= \{(\X(s), \Y(s))\mid s\in \Real\},
\end{equation*}
where
\begin{equation*}
\X(s)=\sup \{\xi\in \Real\mid t(\xi',2s-\xi')<T \quad \text{ for all } \xi'<\xi\},
\end{equation*}
and $\Y(s)=2s-\X(s)$. 

Through the values of $U(\xi, \eta)$ along $\mathcal{C}_T$, we can define the function $u(T,X)$ as follows
\begin{equation*}
u(T,X)= U(\X(s), \Y(s)) \quad \text{ if } x(\X(s), \Y(s))= X.
\end{equation*}
For the functions $R$ and $S$, we can use, in analogy to \eqref{backtoEuler}, that 
\begin{align*}
R(T, X) x_\xi(\X(s), \Y(s))&=c(u(T,X))U_\xi(\X(s), \Y(s)) \\& \quad \text{ for any } s \text{ such that } x(\X(s), \Y(s))= X \text{ and } \dot \X(s)>0, 
\end{align*}
and 
\begin{align*}
S(T,X)x_\eta(\X(s), \Y(s))& =- c(u(T,X))U_\eta(\X(s), \Y(s))\\
& \quad \text{ for any } s \text{ such that } x(\X(s), \Y(s))= X \text{ and } \dot \Y(s)>0 .
\end{align*}

In view of \eqref{rel:LagrBreak}, one has that if $x_\xi (\bar \xi , \eta)  \to 0$ as $(\bar \xi, \eta) \to (\bar \xi, \bar \eta)\in \mathcal{C}_T$, then 
\begin{equation*}
R^2(\bar t,\bar x)\to \infty  \quad \text{ as } (\bar t, \bar x)=(t, x)(\bar \xi , \eta) \to (t,x)(\bar \xi, \bar \eta)=(T,X).
\end{equation*}
If, on the other hand, $x_\eta ( \xi , \bar \eta)  \to 0$ as $(\xi, \bar \eta) \to (\bar \xi, \bar \eta)\in \mathcal{C}_T$, then 
\begin{equation*}
S^2(\bar t,\bar x)\to \infty  \quad \text{ as } (\bar t, \bar x)=(t, x)(\bar \xi , \eta) \to (t,x)(\bar \xi, \bar \eta)=(T,X).
\end{equation*}
This means that if wave breaking occurs in Lagrangian coordinates, which corresponds to $x_\xi=0$ or $x_\eta=0$, then either $R^2$ or $S^2$ in Eulerian coordinates blow up. 

To complete the tuplet $(u,R,S, \mu, \nu)$, it remains to extract the measures $\mu$ and $\nu$. Therefore, introduce $x(s)= x(\X(s), \Y(s))$ for all $s\in \Real$. Then, 
\begin{equation*}
\mu= x_{\#}(J_\xi (\X(s), \Y(s))\dot \X(s)ds)
\end{equation*}
and 
\begin{equation*}
\nu=x_{\#} (J_\eta(\X(s), \Y(s)) \dot \Y(s)ds).
\end{equation*}
Here we want to highlight that $\mu$ and $\nu$ are positive Radon measure, since $Z_\xi$ and $Z_\eta$ not only satisfy  \eqref{rel:tx} and \eqref{signt}, but also 
\begin{equation*}
J_\xi(\xi, \eta)\geq 0 \quad \text{ and } \quad J_ \eta(\xi, \eta)\geq 0,
\end{equation*}
which are conserved in the following sense 
\begin{equation}\label{cons:meas}
(\mu+\nu)(T, \Real)=(\mu_0+\nu_0)(\Real).
\end{equation}
Furthermore, one has as earlier, for $\tilde{S}=x(\tilde{\mathcal{B}})$, where $\tilde{\mathcal{B}}= \{s\mid \dot x(s)=0\}$, that 
\begin{equation*}
(\mu+\nu)(T)_{\mathrm{ac}}= (\mu+\nu)(T)\vert_{\tilde {\mathcal{S}}^c} \quad \text{ and }\quad (\mu+\nu)(T)_{\mathrm{sing}}= (\mu+\nu)(T)\vert_ {\tilde {\mathcal{S}}}.
\end{equation*}

Last but not least we want to highlight and prove one property of the solution $u(t,x)$, which is essential in the next sections.

\begin{lemma}\label{lem:hol}
The function $u(t,x)$ is globally H{\"o}lder continuous, i.e., there exists a constant $D$, dependent on $\kappa$ and $(\mu_0+\nu_0)(\Real)$, such that 
\begin{equation*}
\vert u(t_1,x_1)-u(t_2,x_2)\vert \leq D\sqrt{\vert t_2-t_1\vert + \vert x_2-x_1\vert } \quad \text{ for all } (t_1,x_1), (t_2,x_2)\in\mathbb{R}^2.
\end{equation*}
\end{lemma}

\begin{proof}
Let $(t_1,x_1)$, $(t_2,x_2)\in \mathbb{R}^2$ such that $t_1<t_2$. Then there exist $(\xi_1, \eta_1)$ and $(\xi_2, \eta_2)$ in $\mathbb{R}^2$, not necessarily unique, such that 
\begin{equation*}
(t_i, x_i, u(t_i, x_i))= (t,x,U)(\xi_i, \eta_i) \quad i\in \{1,2\}.
\end{equation*}
Furthermore, there exists a point $(\bar \xi, \bar \eta)\in \mathbb{R}^2$ such that 
\begin{equation}\label{cond:3pun}
\bar \xi + \bar \eta=B= \xi_2+ \eta_2 \quad \text{ and } \quad (t, x, U)(\bar \xi, \bar \eta)= (t_1, \bar x, u(t_1, \bar x)),
\end{equation}
which implies that 
\begin{align*}
\vert x_1- \bar x\vert & \leq \vert x_1-x_2\vert + \vert x_2-\bar x\vert= \vert x_1-x_2\vert +\vert  x(\xi_2,\eta_2)- x(\bar\xi, \bar \eta)\vert.
\end{align*}
By \eqref{cond:c}, \eqref{rel:tx}, \eqref{signt}, and \eqref{cond:3pun}, the last term on the right hand side satisfies
\begin{align*}
\vert x(\xi_2, \eta_2)- x(\bar \xi, \bar \eta)\vert & = \vert \int_{\bar \xi}^{\xi_2} (x_\xi-x_\eta)(\tilde \xi, B-\tilde \xi ) d\tilde \xi \vert\\
&  \leq \kappa\vert \int_{\bar \xi}^{\xi_2} (t_\xi-t_\eta)(\tilde \xi, B-\tilde \xi) d\tilde \xi\vert \\
 & = \kappa \vert t(\xi_2, \eta_2)- t(\bar \xi, \bar \eta) \vert = \kappa \vert t_2-t_1\vert
\end{align*}
and 
\begin{equation*}
\vert x_1-\bar x\vert \leq \kappa\vert t_2- t_1\vert + \vert x_2-x_1\vert.
\end{equation*} 

To obtain the anticipated H{\"o}lder estimate, write 
\begin{equation}\label{est:hol1}
\vert u(t_1, x_1)-u(t_2, x_2)\vert  \leq \vert u(t_1, x_1)-u(t_1, \bar x)\vert+ \vert u(t_1, \bar x)- u(t_2, x_2)\vert.
\end{equation}
For the first term on the right hand side, observe that by Definition~\ref{def:eul}
\begin{equation*}
2c(u)u_x(t_1, x)= (R-S)(t_1, x) \quad \text{ for all } x\in \Real 
\end{equation*}
and hence
\begin{align}\nonumber
\vert u(t_1, x_1)- u(t_1, \bar x)\vert & = \vert \int_{\bar x}^{x_1}  \frac{1}{2c(u)}(R-S)(t_1, \tilde x) d\tilde x\vert \\ \nonumber
&  \leq 2\kappa \sqrt{(\mu+\nu)(t_1, \Real)}\sqrt{\vert \bar x- x_1\vert }\\ \label{est:hol2}
& \leq  2\kappa\sqrt{(\mu_0+\nu_0)( \Real)} \sqrt{\kappa \vert t_2-t_1\vert +\vert x_2- x_1\vert },
\end{align}
where we also used \eqref{cons:meas}.

For the second term on the right hand side of \eqref{est:hol1}, note that since $t_1<t_2$, \eqref{cond:3pun} and \eqref{signt} imply that 
\begin{equation*} 
\bar \xi <\xi_2, \quad \bar \eta>\eta_2, \quad \text{ and } \quad t_1\leq t(\bar \xi, \eta_2)\leq t_2.
\end{equation*}
Combining next the above relations with \eqref{rel:seebreak}, \eqref{rel:tx}, \eqref{signt}, \eqref{asymp:J1}, and \eqref{asymp:J2} yields 
\begin{align*}
\vert U(\bar \xi, \bar \eta)- U(\bar \xi, \eta_2)\vert&  \leq \sqrt{2\kappa} \int_{ \eta_2}^{\bar \eta} \sqrt{(-t_\eta J_\eta)(\bar \xi , \tilde \eta)} d\tilde \eta\\
& \leq \sqrt{2\kappa} \sqrt{\vert J(\bar \xi, \bar\eta)- J(\bar \xi, \eta_2)\vert } \sqrt{\vert t(\bar \xi, \eta_2)- t(\bar \xi , \bar \eta)\vert } \\\
& \leq \sqrt{2\kappa}  \sqrt{(\mu_0+\nu_0)(\Real)}\sqrt{\vert t_2-t_1\vert } 
\end{align*}
and following the same lines 
\begin{equation*}
\vert U(\bar \xi, \eta_2)- U(\xi_2, \eta_2)\vert \leq \sqrt{2\kappa} \sqrt{(\mu_0+\nu_0)(\Real)}\sqrt{\vert t_2-t_1\vert }.
\end{equation*}
Thus 
\begin{align*}
\vert u(t_1, \bar x)- u(t_2, x_2)\vert& \leq \vert U(\bar \xi, \bar \eta)- U(\bar \xi, \eta_2)\vert + \vert U(\bar \xi, \eta_2)- U(\xi_2, \eta_2)\vert \\
& \leq \sqrt{8\kappa} \sqrt{(\mu_0+\nu_0)(\Real)} \sqrt{\vert t_2-t_1\vert},
\end{align*}
which combined with \eqref{est:hol1} and \eqref{est:hol2} finishes the proof. 
\end{proof}

\section{Prediction of wave breaking}

In the last section, we outlined that wave breaking occurs along backward characteristics, which correspond to vertical lines in the $(\xi, \eta)$ plane, if $x_ \xi(\xi, \eta)$ tends to zero, while wave breaking occurs along forward characteristics, which correspond to horizontal lines in the $(\xi, \eta)$ plane, if $x_\eta(\xi, \eta)$ tends to zero. In this section we will present some criteria, which allow to predict whether or not wave breaking either will take place in the near future or occurred recently. Again, we will distinguish between following backward and forward characteristics to differ between $R(t,x)=(u_t+c(u)u_x)(t,x)$ and $S(t,x)=(u_t-c(u)u_x)(t,x)$ becoming unbounded pointwise. 

\subsection{Wave breaking along backward characteristics}\label{sub:BC} Let $(\xi_0, \eta_0)\in \Real^2$, which corresponds to the point $(0,\bar x)= (t(\xi_0, \eta_0), x(\xi_0, \eta_0))$ in Eulerian coordinates, such that $x_\xi (\xi_0, \eta_0)>0$. Then wave breaking occurs along the backward characteristic $y(t,\xi_0)$ if the function $R(t,y(t,\xi_0))=(u_t+c(u)u_x)(t,y(t,\xi_0))$ becomes unbounded, which, as highlighted in the last section, is equivalent to $x_\xi(\xi_0, \eta)$ tending to $0$ along the vertical line $\{(\xi_0, \eta)\mid \eta\in \Real\}$ in the $(\xi, \eta)$ plane. A natural question that arises in this context is the following: Does there exist a point $(\xi_0, \bar \eta)$ close to $(\xi_0, \eta_0)$ such that 
$x_\xi (\xi_0, \eta)\to 0$ as $\eta\to \bar \eta$? If so, wave breaking either occurred recently or will take place in the near future and it then remains to estimate the actual breaking time. 

As a starting point we consider the function 
\begin{equation}\label{def:f}
  f(\xi_0, \eta)= \sqrt{\frac{c(U(\xi_0, \eta_0))}{c(U(\xi_0, \eta))}}x_\xi (\xi_0, \eta). 
\end{equation}
Here two observations are important. First, due to \eqref{cond:c}, $f(\xi_0,\eta)\to 0$ as $\eta\to \bar \eta$ is equivalent to $x_\xi(\xi_0, \eta) \to 0$ as $\eta\to \bar \eta$. Second, $f(\xi_0, \eta)$ satisfies, in contrast to $x_\xi(\xi_0, \eta)$, cf. \eqref{eq:Lagr}, a homogeneous linear differential equation. Namely,
\begin{equation*}
  f_\eta(\xi_0, \eta)= \frac{c'(U)}{2c(U)}\frac{U_\xi}{x_\xi}x_\eta f(\xi_0, \eta),
\end{equation*}
which implies that 
\begin{equation*}
  f(\xi_0, \eta)= f(\xi_0, \eta_0) e^{\int_{\eta_0}^\eta \frac{c'(U)}{2c(U)} \frac{U_\xi}{x_\xi} x_\eta(\xi_0, \tilde \eta) d\tilde \eta}.
\end{equation*}
By \eqref{def:f} $f(\xi_0, \eta_0)>0$, since, by assumption, $x_\xi(\xi_0, \eta_0)>0$ and $c(x)$ satisfies \eqref{cond:c}, and hence $f(\xi_0, \eta)$ can only tend to zero as $\eta\to \bar \eta$ if 
\begin{equation}\label{asymp}
  \int_{\eta_0}^\eta \frac{c'(U)}{2c(U)} \frac{U_\xi}{x_\xi} x_\eta(\xi_0, \tilde \eta) d\tilde \eta \rightarrow -\infty.
\end{equation}
Thus our goal is to establish a criterion, which guarantees that the above integral tends to $-\infty$ for $\bar \eta$ close to $\eta_0$, which can be formulated in Eulerian coordinates. The key element of our argument is the function
\begin{equation} \label{def:h}
  h(\xi_0, \eta)= c(U)\frac{U_\xi}{x_\xi}(\xi_0, \eta),
\end{equation}
since \eqref{asymp} can only hold if $c'(U)h(\xi_0, \eta)$ becomes unbounded, due to \eqref{cond:c} and \eqref{prop:rel}. Moreover, by \eqref{cond:cder} $c'$ is bounded but can attain the value $0$, so that 
$c'(U)h(\xi_0, \eta)$ can only become unbounded if $h(\xi_0, \eta)$ becomes unbounded. In a next step, we will therefore analyse when $h(\xi_0, \eta)$ can become unbounded by having a closer look at the underlying first order differential equation, which is given by 
\begin{equation}\label{difflig:h}
  h_\eta(\xi_0, \eta)= \alpha(\xi_0, \eta)+ \gamma h^2(\xi_0, \eta),
\end{equation}
where 
\begin{equation*}
 \alpha(\xi_0, \eta)= \frac{c'(U)}{2c^2(U)}J_\eta(\xi_0, \eta) 
\end{equation*}
and 
\begin{equation}\label{def:tigamma}
  \gamma(\xi_0, \eta)= -\frac{c'(U)}{4c^2(U)}x_\eta(\xi_0, \eta).
\end{equation}
Note that by \eqref{signt}, \eqref{rel:tx}, and \eqref{rel:seebreak}, the signs of $\alpha(\xi_0, \eta)$ and $-\gamma(\xi_0, \eta)$ coincide with the sign of $c'(U)(\xi_0, \eta)$

\vspace{0.2cm}
Assume that $\eta_0\leq  \eta$, which implies, using \eqref{signt}, that $t(\xi_0, \eta_0)\geq t(\xi_0, \eta)$ and hence will yield a criterion for recent wave breaking. Then \eqref{asymp} requires 
\begin{equation*}
c'(U)h(\xi_0, \eta) \to -\infty \quad \text{ as }\quad  \eta\to \bar \eta, 
\end{equation*}
for some $\bar \eta \geq \eta_0$. Therefore, we will distinguish two cases dependent on the signs of $c'(U)(\xi_0, \eta)$ and $h(\xi_0, \eta)$. 

{\it Case 1:}
 Assume there exists an interval $[\eta_0, \eta_1]$ such that 
\begin{equation*}
  h(\xi_0, \eta)<0 \quad \text{ and } \quad c'(U)(\xi_0, \eta)>0 \quad \text{ for all }\eta\in [\eta_0, \eta_1],
\end{equation*}
and hence 
\begin{equation}\label{cond:a1}
  \alpha(\xi_0, \eta) \geq 0 \quad \text{ and }\quad  \gamma(\xi_0, \eta) \leq 0 \quad \text{ for all } \eta\in [\eta_0, \eta_1].
\end{equation}
Then \eqref{difflig:h} implies 
\begin{equation*}
  h(\xi_0,\eta_0)+ \int_{\eta_0}^\eta \gamma h^2(\xi_0, \tilde \eta) d\tilde \eta\leq  h(\xi_0, \eta)\leq a+ \int_{\eta_0}^\eta \gamma h^2(\xi_0, \tilde \eta) d\tilde \eta
\end{equation*}
for any real number $a$ which satisfies 
\begin{equation}\label{cond:a3}
  h(\xi_0, \eta_0)+ \int_{\eta_0}^{\eta_1}  \alpha(\xi_0, \tilde \eta) d\tilde \eta\leq a<0.
\end{equation} 
Furthermore, the above integral inequality fulfills all the assumptions from Proposition~\ref{prop:ber} due to \eqref{cond:a1} and we end up with  
\begin{equation*}
  \frac{h(\xi_0, \eta_0)}{1- h(\xi_0, \eta_0)\int_{\eta_0}^\eta \gamma(\xi_0, \tilde \eta) d\tilde \eta}\leq h(\xi_0, \eta) \leq \frac{a}{1- a\int_{\eta_0}^\eta  \gamma(\xi_0, \tilde \eta) d\tilde \eta},
\end{equation*}
where $ a$ can be any real number which satisfies \eqref{cond:a3}.

{\it Case 2:} Assume there exists an interval $[\eta_0, \eta_1]$ such that 
\begin{equation*}
  h(\xi_0, \eta)>0 \quad \text{ and } \quad c'(U)(\xi_0, \eta)<0 \quad \text{ for all }\eta\in [\eta_0, \eta_1],
\end{equation*}
and hence 
\begin{equation*}
  \alpha(\xi_0, \eta) \leq 0 \quad \text{ and }\quad \gamma(\xi_0, \eta) \geq 0 \quad \text{ for all } \eta\in [\eta_0, \eta_1].
\end{equation*}
Then $\check h(\xi_0, \eta)= -h(\xi_0, \eta)$ satisfies the differential equation
\begin{equation*}
  \check h_\eta(\xi_0, \eta)= \check \alpha(\xi_0, \eta)+\check\gamma \check h^2(\xi_0, \eta), 
\end{equation*}
where 
\begin{equation*}
  \check \alpha(\xi_0, \eta)=- \alpha(\xi_0, \eta)\geq 0
\end{equation*}
 and 
\begin{equation*}
  \check \gamma(\xi_0, \eta)=- \gamma (\xi_0, \eta) \leq 0 
\end{equation*}
for all $\eta\in [\eta_0, \eta_1]$. This means $\check h(\xi_0, \eta)$ satisfies all the assumptions from Case 1 and 
we obtain 
\begin{equation*}
  \frac{\check h(\xi_0, \eta_0)}{1- \check h(\xi_0, \eta_0) \int_{\eta_0}^\eta \check \gamma(\xi_0, \tilde \eta) d\tilde \eta} \leq \check h(\xi_0, \eta)\leq \frac{\check a}{1- \check a \int_{\eta_0}^\eta \check \gamma (\xi_0, \tilde \eta) d\tilde \eta},
\end{equation*}
where $\check a$ can be any real number which satisfies
\begin{equation*}
  \check h(\xi_0, \eta_0)+ \int_{\eta_0}^{\eta_1} \check \alpha(\xi_0, \tilde\eta) d\tilde \eta\leq \check a<0.
\end{equation*}
Thus we end up with 
\begin{equation*}
  \frac{h(\xi_0, \eta_0)}{1-h(\xi_0, \eta_0)\int_{\eta_0}^\eta  \gamma(\xi_0, \tilde \eta) d\tilde \eta}\geq h(\xi_0, \eta) \geq \frac{a} {1- a \int_{\eta_0}^\eta  \gamma (\xi_0, \tilde \eta) d\tilde \eta},
\end{equation*}
where $a$ can be any real number such that 
\begin{equation}\label{cond:a4}
  0<a \leq h(\xi_0, \eta_0)+ \int_{\eta_0}^{\eta_1} \alpha (\xi_0, \tilde \eta) d\tilde \eta.
\end{equation}

To finally obtain an upper and a lower bound on \eqref{asymp}, note that 
\begin{equation*}
\frac{c'(U)}{2c(U)}\frac{U_\xi}{x_\xi}x_\eta(\xi_0, \eta)= -2\gamma h(\xi_0, \eta) \leq 0 \quad \text{ for  all } \eta \in [\eta_0, \eta_1],
\end{equation*}
and $-\gamma(\xi_0, \eta)$ and $h(\xi_0, \eta)$ have opposite sign. Thus, combining Case 1 and Case 2 yields  
\begin{align}\nonumber
  \int_{\eta_0}^\eta \frac{c'(U)}{2c(U)} \frac{U_\xi}{x_\xi} x_\eta(\xi_0, \tilde \eta)d \tilde \eta 
  & = -2 \int_{\eta_0}^\eta  \gamma h(\xi_0, \tilde \eta) d\tilde \eta\\ \nonumber
  & \geq -2 \int_{\eta_0}^\eta \frac{h(\xi_0, \eta_0) \gamma (\xi_0, \tilde \eta)}{1-h(\xi_0, \eta_0)\int_{\eta_0}^{\tilde \eta} \gamma (\xi_0, \hat \eta) d\hat \eta}d \tilde \eta\\ \nonumber
  & = 2 \ln\left(1- h(\xi_0, \eta_0) \int_{\eta_0}^\eta \gamma(\xi_0, \tilde \eta) d\tilde \eta\right)\\ \nonumber
  & = 2\ln\left(1+h(\xi_0, \eta_0)\int_{\eta_0}^\eta \frac14 \frac{c'(U)}{c^2(U)}x_\eta(\xi_0, \tilde\eta) d\tilde \eta\right)\\ \label{est:bound1}
  & = 2 \ln\left(1- h(\xi_0, \eta_0) \int_{\eta_0}^\eta \frac 14 \frac{c'(U)}{c(U)}t_\eta(\xi_0, \tilde \eta)d\tilde \eta \right),
\end{align}
where we used \eqref{def:tigamma}
and, following the same lines, 
\begin{equation}\label{est:bound2}
 \int_{\eta_0}^\eta \frac{c'(U)}{2c(U)} \frac{U_\xi}{x_\xi} x_\eta(\xi_0, \tilde \eta)d \tilde \eta  \leq  2\ln\left(1- a \int_{\eta_0}^\eta \frac14 \frac{c'(U)}{c(U)} t_\eta (\xi_0, \tilde \eta) d\tilde \eta\right).
\end{equation}
Comparing the above inequalities \eqref{est:bound1} and \eqref{est:bound2}, we observe that they have nearly the same form. To ensure that wave breaking takes place and to obtain an upper and a lower bound on the wave breaking time, we will bound the constant $a$ by a multiple of $h(\xi_0, \eta_0)$. This will be done next. 

A closer look at \eqref{cond:a3} and \eqref{cond:a4} reveals that $a$,  $h(\xi_0, \eta_0)$,  and $-\alpha(\xi_0, \eta)$ for $\eta\in [\eta_0, \eta_1]$ have the same sign. Thus, all constants $a$, which satisfy
\begin{equation}\label{cond:a5}
0<\vert a\vert \leq \vert h(\xi_0, \eta_0)\vert - \int_{\eta_0}^{\eta_1} \vert \alpha(\xi_0, \eta) \vert d\eta,
\end{equation}
are admissible. 
Furthermore, \eqref{cond:c}, \eqref{cond:cder}, \eqref{asymp:J1}, and \eqref{asymp:J2} imply that 
\begin{equation*}
\int_{\eta_0}^{\eta_1} \vert \alpha(\xi_0, \eta)\vert d\eta\leq \frac12 \lambda \kappa^2 (J(\xi_0, \eta_1)-J(\xi_0, \eta_0))\leq  \frac12 \lambda \kappa^2 (\mu_0+\nu_0)(\Real)
\end{equation*}
and \eqref{cond:a5} is satisfied for all constants $a$ such that 
\begin{equation}\label{cond:a6}
0<\vert a\vert \leq \vert h(\xi_0, \eta_0)\vert - \frac12 \lambda \kappa^2 (\mu_0+\nu_0)(\Real).
\end{equation}
This means in particular, that if there exists a constant $A\geq 1$ such that 
\begin{equation}\label{cond:HA}
\vert h(\xi_0, \eta_0)\vert = A\lambda \kappa^2 (\mu_0+\nu_0)(\Real),
\end{equation}
then the constant $a$ given by 
\begin{equation}\label{def:cona}
a= \left (1-\frac{1}{2A}\right) h(\xi_0, \eta_0), 
\end{equation}
will not only satisfy \eqref{cond:a6}, but also \eqref{cond:a3} and \eqref{cond:a4}, which require that $a$ and $h(\xi_0, \eta_0)$ have the same sign, and \eqref{est:bound2} rewrites as 
\begin{align}\label{est:bound33}
  \int_{\eta_0}^\eta \frac{c'(U)}{2c(U)}\frac{U_\xi}{x_\xi}x_\eta(\xi_0, \tilde \eta) d\tilde \eta
 \leq 2 \ln \left(1- \left(1-\frac1{2A}\right) h(\xi_0, \eta_0)\int_{\eta_0}^\eta \frac14 \frac{c'(U)}{c(U)} t_\eta(\xi_0, \tilde \eta) d\tilde \eta\right).
\end{align}
Recalling that $\eta\geq \eta_0$ and  \eqref{asymp}, which must be satisfied if wave breaking takes place along a backward characteristic, yields that wave breaking occurred for sure backward in time if there exists $\bar \eta>\eta_0$ such that 
 \begin{equation*}
1- \left(1-\frac1{2A}\right) h(\xi_0, \eta_0)\int_{\eta_0}^\eta \frac14 \frac{c'(U)}{c(U)} t_\eta(\xi_0, \tilde \eta) d\tilde \eta \to 0 \quad \text{ as } \quad \eta \to \bar \eta
 \end{equation*}
 and \eqref{est:bound33} and  \eqref{est:bound1} will provide a lower and an upper bound on the wave breaking time, respectively. 
 
\begin{theorem}[Wave breaking along backward characteristics - Part 1]\label{thm:wb1}
Given some initial data $(u_0, R_0, S_0, \mu_0, \nu_0)\in \mathcal{D}$, denote by $(u(t), R(t), S(t), \mu(t), \nu(t))$ the global conservative solution to the NVW equation at time $t$ and by $D$ the constant from Lemma~\ref{lem:hol}. If the following conditions are satisfied
\begin{itemize}
\item  $c'(u_0)R_0(\bar x)=c'(u_0)(u_{0,t}+c(u_0)u_{0,x})(\bar x)<0$, 
\item there exists $A\geq 1$ such that 
\begin{equation*}
  \vert R_0\vert (\bar x)= A \lambda \kappa^2 (\mu_0+ \nu_0)(\Real) \quad \text{ and }
\end{equation*}
\item 
$12 \bar \lambda^2D^2(1+\kappa) <(1-\frac{1}{2A})\frac{1}{\kappa} \vert c'(u_0)\vert^3\vert R_0\vert(\bar x)$,
\end{itemize}
then wave breaking occurred along the backward characteristic $y(t, \xi_0)$ with $y(0,\xi_0)=\bar x$ within the time interval 
 \begin{equation*}
[t_l, t_u]=[-\frac{3}{b},  - \frac{3}{5\tilde b}], 
\end{equation*}
where 
\begin{equation*}
b= \frac{1}{4\kappa}  \left(1- \frac{1}{2A}\right)\vert c'(u_0)R_0\vert (\bar x)< \frac{1}{4}\kappa \vert c'(u_0) R_0\vert ( \bar x)= \tilde b .
\end{equation*}
\end{theorem}

\begin{proof} Assume without loss of generality that 
\begin{equation*}
R_0(\bar x)<0\quad \text{ and }\quad c'(u_0)(\bar x)>0,
\end{equation*}
since the other case can be treated similarly. Then there exists $(\xi_0, \eta_0)$, not necessarily unique, such that 
\begin{equation*}
  (0,\bar x)=(0, y(0, \xi_0))= (t(\xi_0, \eta_0), x(\xi_0, \eta_0)) \quad \text{ and } \quad R_0(\bar x)= h (\xi_0, \eta_0).
\end{equation*}
Furthermore, the integral turning up in both \eqref{est:bound1} and \eqref{est:bound33}, can be rewrittes as 
\begin{align*}
  \int_{\eta_0}^\eta \frac{c'(U)}{c(U)} t_\eta(\xi_0, \tilde \eta) d\tilde \eta
  & = \int_{\eta_0}^\eta \frac{c'(u(t(\xi_0, \tilde \eta), x(\xi_0, \tilde \eta)))}{c(u(t(\xi_0, \tilde \eta), x(\xi_0, \tilde \eta)))}t_\eta(\xi_0,\tilde \eta) d\tilde \eta\\
  & =- \int_{t(\xi_0, \eta)}^{0} \frac{c'(u(s, y(s, \xi_0)))}{c(u(s, y(s,\xi_0)))} ds,
\end{align*}
since $x(\xi_0, \eta)= y(t(\xi_0, \eta), \xi_0)$ for all $\eta\in \Real$. Thus, we aim at estimating the integral on the right hand side from above as well as below. However, be aware that the estimates \eqref{est:bound1} and \eqref{est:bound33} are only valid as long as $c'(u(s, y(s, \xi_0)))>0$ for all $s\in (t(\xi_0, \eta), 0]$. This is guaranteed locally as we will see next. 

By \eqref{cond:cder}, Lemma~\ref{lem:hol} and \eqref{def:char}, we have that 
\begin{align}\nonumber
c'(u(s, y(s, \xi_0)))& \geq c'(u_0(y(0, \xi_0)))- \vert c'(u(s, y(s, \xi_0)))- c'(u(0, y(0, \xi_0)))\vert \\ \nonumber
& \geq c'(u_0(\bar x))- \bar \lambda D \sqrt{\vert s\vert + \vert y(s, \xi_0)- y(0, \xi_0)\vert }\\ \label{cp:lb}
& \geq c'(u_0( \bar x))- \bar \lambda D \sqrt{1+ \kappa} \sqrt{\vert s\vert}>0
\end{align}
for all $0\leq -s\leq -t(\xi_0, \eta)<  T$, 
where $T$ satisfies 
\begin{equation}\label{def:T}
\vert c'(u_0( \bar x))\vert = \bar\lambda D \sqrt{1+\kappa} \sqrt{T}.
\end{equation} 
 Furthermore, following the same lines we also obtain an upper bound
\begin{equation}\label{cp:ub}
c'(u(s, y(s, \xi_0)))\leq c'(u_0( \bar x))+\bar \lambda D \sqrt{1+ \kappa} \sqrt{\vert s\vert}.
\end{equation}

Assume from now on that $0\leq -t(\xi_0, \eta)\leq T$, then 
\begin{align} \nonumber
 \int_{\eta_0}^\eta \frac{c'(U)}{c(U)} t_\eta(\xi_0, \tilde \eta) d\tilde \eta& = - \int_{t(\xi_0, \eta)}^{0} \frac{c'(u(s, y(s, \xi)))}{c(u(s, y(s,\xi)))} ds\\ \nonumber
  & \leq \frac{1}{\kappa}\int_{t(\xi_0, \eta)}^{0} -c'(u_0(\bar x))+\bar \lambda D\sqrt{1+\kappa} \sqrt{-s} ds\\ \label{cond:integral}
  & =\frac{1}{\kappa} \left(c'(u_0(\bar x))t(\xi_0, \eta)+\frac23 \bar \lambda D \sqrt{1+\kappa}(-t(\xi_0, \eta))^{3/2}\right).
 \end{align}
To obtain a lower bound $t_l$ for the exact wave breaking time $t(\xi_0,\bar  \eta)$, observe that \eqref{est:bound33} combined with the above estimate implies that $t_1\leq t(\xi_0, \bar \eta)<0$, where $t_1$ is implicitly given as the solution to 
\begin{equation*}
  1- \frac14\left(1- \frac{1}{2A}\right)R_0( \bar x)\frac{1}{\kappa} \left(c'(u_0(\bar x))t_1+\frac23 \bar \lambda D \sqrt{1+\kappa}(-t_1)^{3/2}\right)=0.
\end{equation*}
Note that those $(-t_1)^{1/2}>0$, which satisfy the above equality, can bee seen as zeros (on the positive real line) of the third order polynomial 
\begin{equation*}
  f(x)= -ax^3+ bx^2-1,
\end{equation*}
where $a$ and $b$ are positive constants. As a closer look reveals $f$ attains a minimum at $x=0$, where $f(0)=-1$, a maximum at $x= \frac{2b}{3a}=\sqrt{T}$, where $f(\frac{2b}{3a})=\frac{4b^3}{27 a^2}-1$, and is strictly decreasing on the interval $[\frac{2b}{3a}, \infty)$. Thus $f(x)$ will either have $0$, $1$, or $2$ positive zeros. In particular, one has that the change from $0$ to $2$ positive zeros $\bar z_1 < \bar z_2$ occurs when $4b^3= 27 a^2$. Or, in other words for $4b^3>27 a^2$, which we assume from now on, $f(x)$ has two positive zeros $\bar z_1<\frac{2b}{3a}=\sqrt{T}<\bar z_2$ and wave breaking occurs before $c'(u(s, y(s, \xi_0)))$ changes sign. Although we cannot compute $\bar z_1=(-t_1)^{1/2}$ explicitly, we can try to estimate it from above. Therefore, observe that 
 \begin{equation*}
 g(x)= \frac13 bx^2-1 \leq f(x) \quad \text{ for all } x\in [0, \frac{2b}{3a}],
 \end{equation*}
 with equality for $x=0$ and $x= \frac{2b}{3a}$. Thus there exists a unique $\hat z_1\in(0, \frac{2b}{3a}]$ such that $g(\hat z_1)=0$ and $z_1\leq \hat z_1$. As a closer look reveals $\hat z_1=\sqrt{ \frac{3}{b}}$ and 
 \begin{equation*}
 t_l=- \frac{3}{b}= -\hat z_1^2\leq - z_1^2 =t_1.
 \end{equation*}
 
To obtain an upper bound $t_u$ for the exact wave breaking time $t(\xi_0, \bar \eta)$, we slightly modify the steps leading to \eqref{cond:integral} by using \eqref{cp:ub} to obtain the following lower bound
\begin{equation*}
 \int_{\eta_0}^{\eta} \frac{c'(U)}{c(U)} t_\eta(\xi, \tilde \eta) d\tilde \eta\geq \kappa\left(c'(u_0( \bar x))t(\xi_0, \eta)-\frac23 \bar\lambda D \sqrt{1+\kappa}(-t(\xi_0, \eta))^{3/2}\right).
\end{equation*}
Combined with \eqref{est:bound1} this estimate implies that $t(\xi_0, \bar \eta)\leq t_2<0$, where $t_2$ is implicitly given as the solution to 
\begin{equation*}
  1- \frac14R_0( \bar x)\kappa \left(c'(u_0(\bar x))t_2-\frac23 \bar \lambda D \sqrt{1+\kappa}(-t_2)^{3/2}\right)=0.
\end{equation*}
Note that those $(-t_2)^{1/2}>0$, which satisfy the above equality, can be seen as zeros (on the positive real line) of the third order polynomial  
\begin{equation*}
  \tilde f(x)= \tilde ax^3+\tilde bx^2-1,
\end{equation*}
where $\tilde a$ and $\tilde b$ are positive constants. As a closer look reveals $\tilde f$ attains a minimum at $x=0$, where $\tilde f(0)=-1$, is strictly increasing on the interval $[0, \infty)$ and crosses the positive $x$-axis in exactly one point $\bar z$, which we can try to estimate from below. Therefore, observe that 
\begin{equation*}
\tilde a = \frac{1}{1- \frac{1}{2A}}\kappa^2a \geq a \quad \text{ and }\quad  \tilde b= \frac{1}{1- \frac{1}{2A}} \kappa^2 b \geq b,
\end{equation*}
 which implies $\frac{2b}{3a}= \frac{2\tilde b}{3\tilde a}=\sqrt{T}$  and that 
\begin{equation*}
\tilde g(x)= \frac{5}{3} \tilde b x^2-1 \geq \tilde f(x)\geq f(x) \quad \text{ for all } x\in [0, \frac{2\tilde b}{3\tilde a}]
\end{equation*}
with $\tilde g(x)= \tilde f(x)$ for $x=0$ and $x= \frac{2\tilde b}{3\tilde a}$. Thus there exists a unique $\hat z \in (0, \frac{2\tilde b}{3\tilde a})$ such that $\tilde g(\hat z)=0$ and $\hat z\leq \bar z$. As a closer look reveals $\hat z= \sqrt{\frac{3}{5\tilde b}}$ and 
\begin{equation*}
t_2= - \bar z^2\leq - \hat z^2 =- \frac{3}{5\tilde b}= t_u.
\end{equation*}

This finishes the proof, since the wave breaking time $t(\xi_0, \bar\eta)$ satisfies 
\begin{equation*}
-\frac{3}{b}=t_l\leq t_1\leq t(\xi_0, \bar \eta)\leq t_2\leq t_u= -\frac{3}{5\tilde b}.
\end{equation*} 
\end{proof}

Next, assume that  $\eta\leq \eta_0$, which implies, using \eqref{signt}, that $t(\xi_0, \eta_0) \leq t(\xi_0, \eta)$ and hence will yield a criterion for wave breaking in the near future. In this case \eqref{asymp} requires
\begin{equation}\label{limit:1}
c'(U)h(\xi_0, \eta) \to \infty \quad \text{ as } \quad \eta\to \bar \eta
\end{equation}
for some $\bar \eta\leq \eta_0$. Due to a symmetry argument, which we present next, the result is an immediate consequence of the estimates established until now. 

Define the change of coordinates
\begin{equation*}
  (\xi,\eta) \to (2\xi_0- \xi, 2\eta_0- \eta),
\end{equation*}
which corresponds to rotating $\mathbb{R}^2$ by $\pi$ around the point $(\xi_0, \eta_0)$. In particular, one has that the half line through $(\xi_0, \eta_0)$ given by $\{(\xi_0, \eta)\mid \eta\leq \eta_0\}$ is mapped to the half line through $(\xi_0, \eta_0)$ given by $\{(\xi_0, \eta)\mid \eta\geq \eta_0\}$. Moreover, defining 
\begin{align*}
  \hat Z(\xi, \eta)& = (\hat t, \hat x, \hat U, \hat J)( \xi, \eta)\\
  & = (t, x,U, J)(2\xi_0- \xi, 2\eta_0- \eta)= Z(2\xi_0- \xi, 2\eta_0- \eta),
\end{align*}
which satisfies
\begin{equation}\label{sign:change}
  \hat Z_i(\xi,  \eta)= - Z_i(2\xi_0-\xi, 2\eta_0-\eta) \quad \text{ for } i\in \{\xi, \eta\}
\end{equation}
as well as \eqref{eq:Lagr}, yields upon substitution
\begin{align*}
  \int_{\eta_0}^\eta \frac{c'(U)}{2c(U)} \frac{U_\xi}{x_\xi} x_\eta(\xi_0, \tilde \eta) d\tilde \eta
  & = -\int_{\eta_0}^{2\eta_0- \eta} \frac{c'(U)}{2c(U)} \frac{U_\xi}{x_\xi} x_\eta(\xi_0, 2\eta_0-\tilde \eta) d\tilde \eta\\
  & = \int_{\eta_0}^{2\eta_0-\eta} \frac{c'(\hat U)}{2c(\hat U)} \frac{\hat U_\xi}{\hat x_\xi}\hat x_\eta (\xi_0, \tilde \eta)d\tilde \eta.
\end{align*}
By assumption $\eta_0\geq\eta$, which implies $2\eta_0-\eta\geq \eta_0$ and the integral on the right hand side is of the same form as the integral in \eqref{asymp} and has to satisfy the same limit. However, there is one important difference: $\hat x_\eta(\xi_0, \eta)\leq 0$. Therefore, we look at the integrand as the following product
\begin{equation*}
  \frac{c'(\hat U)}{2c(\hat U)} \frac{\hat U_\xi}{\hat x_\xi}\hat x_\eta (\xi_0,\eta)= -2\hat \gamma \hat h (\xi_0, \eta),
\end{equation*}
where 
\begin{equation}\label{def:hh} 
  \hat h(\xi_0, \eta)=-h(\xi_0, 2\eta_0-\eta)= -c(\hat U) \frac{\hat U_\xi}{\hat x_\xi}(\xi_0, \eta)
\end{equation}
and
\begin{equation}
  \hat \gamma(\xi_0, \eta)=\gamma(\xi_0, 2\eta_0-\eta)= \frac{c'(\hat U)}{4c^2(\hat U)} \hat x_\eta(\xi_0, \eta). 
\end{equation}
In this new coordinates, \eqref{limit:1} rewrites as
\begin{equation*}
c'(\hat U)\hat h(\xi_0, \eta) \to -\infty \quad \text{ as } \quad \eta\to \bar \eta
\end{equation*}
for some $\bar \eta\geq \eta_0$. Furthermore, $\hat h(\xi_0, \eta)$ satisfies the first order differential equation
\begin{equation}\label{difflig:hh}
  \hat h_\eta(\xi_0, \eta)= \hat \alpha(\xi_0, \eta)+ \hat \gamma \hat h^2(\xi_0, \eta),
\end{equation}
where 
\begin{equation}
  \hat \alpha(\xi_0, \eta)=\alpha(\xi_0, 2\eta_0-\eta)= -\frac{c'(\hat U)}{2c^2(\hat U)} \hat J_\eta(\xi_0, \eta).
\end{equation}
and $(\hat \alpha, \hat \gamma, \hat h)$ play the same role and have the same properties as $(\alpha, \gamma, h)$ in the case $\eta\geq \eta_0$. As an immediate consequence we obtain that if $h(\xi_0, \eta_0)$ satisfies \eqref{cond:HA}, then 
\begin{equation}\label{est:bound3}
  \int_{\eta_0}^\eta \frac{c'(U)}{2c(U)}\frac{U_\xi}{x_\xi}x_\eta(\xi_0, \tilde \eta) d\tilde \eta\geq 2\ln\left(1+h(\xi_0, \eta_0)\int_{\eta}^{\eta_0} \frac14 \frac{c'(U)}{c(U)} t_\eta(\xi_0, \tilde \eta) d\tilde \eta\right)
\end{equation}
and
\begin{equation}\label{est:bound4}
  \int_{\eta_0}^\eta \frac{c'(U)}{2c(U)}\frac{U_\xi}{x_\xi}x_\eta(\xi_0, \tilde \eta) d\tilde \eta\leq 2\ln\left(1+\left(1-\frac{1}{2A}\right) h(\xi_0, \eta_0)\int_{\eta}^{\eta_0} \frac14 \frac{c'(U)}{c(U)} t_\eta(\xi_0, \tilde \eta) d\tilde \eta\right).
\end{equation}

From here on we follow the proof of Theorem~\ref{thm:wb1} with slight modifications and we assume without less of generality that 
\begin{equation*}
R_0(\bar x)=(u_{0,t}+ c(u_0)u_{0,x})( \bar x)>0 \quad \text{ and } \quad c'(u_0( \bar x))>0.
\end{equation*} 

An upper bound $t_u$ for the exact wave breaking time $t(\xi_0, \bar \eta)$ is given by $t(\xi_0, \bar \eta)\leq t_1$, where $t_1$ is implicitly given as the solution to 
\begin{equation*}
1- \frac14 \left(1- \frac{1}{2A}\right) R_0( \bar x)\frac{1}{\kappa} \left(c'(u_0( \bar x))t_1- \frac23 \bar \lambda D\sqrt{1+\kappa} t_1^{3/2}\right).
\end{equation*}
Again, those $t_1^{1/2}>0$ which satisfy the above equality can be seen as zeros (on the positive real line) of the third order polynomial 
\begin{equation*}
f(x)= - ax^3+bx^2-1,
\end{equation*}
where $a$ and $b$ are positive constants. In particular, one has that if $27 a^2<4b^3$,
then 
\begin{equation*}
t_1\leq \frac{3}{b}= t_u.
\end{equation*}

A lower bound $t_l$ for the exact wave breaking time $t(\xi_0, \bar \eta)$ is given by $0\leq t_2\leq t(\xi_0, \bar \eta)$, where $t_2$ is implicitly given as the solution to 
\begin{equation*}
1- \frac14 R_0(\bar x)\kappa \left(c'(u_0(\bar x))t_2+ \frac23 \bar \lambda D\sqrt{1+\kappa} t_2^{3/2}\right).
\end{equation*}
Note, those $t_2^{1/2}>0$, which satisfy the above equality, can be seen as zeros (on the positive real line) of the third order polynomial 
\begin{equation*}
\tilde f(x)= \tilde ax^3+\tilde bx^2-1,
\end{equation*}
where $\tilde a$ and $\tilde b$ are positive constants. In particular, one has that 
\begin{equation*}
t_l=\frac{3}{5\tilde b}\leq t_2.
\end{equation*}
Since the wave breaking time $t(\xi_0, \bar \eta)$ satisfies 
\begin{equation*}
t_l\leq t_2\leq t(\xi_0, \bar \eta)\leq t_1\leq t_u,
\end{equation*}
we have shown the following result.

\begin{theorem}[Wave breaking along backward characteristics - Part 2]\label{thm:wb2}
Given some initial data $(u_0,R_0, S_0, \mu_0, \nu_0)\in \mathcal{D}$, denote by $(u(t), R(t),S(t),\mu(t), \nu(t))$ the global conservative solution to the NVW equation at time $t$ and by $D$ the constant from Lemma~\ref{lem:hol}. If the following conditions are satisfied
\begin{itemize}
\item  $c'(u_0)R_0(\bar x)=c'(u_0)(u_{0,t}+c(u_0)u_{0,x})(\bar x)>0$, 
\item there exists $A\geq 1$ such that 
\begin{equation*}
  \vert R_0\vert ( \bar x)= A \lambda \kappa^2 (\mu_0+ \nu_0)(\Real) \quad \text{ and }
\end{equation*} 
\item 
$12 \bar \lambda^2D^2(1+\kappa) <(1-\frac{1}{2A})\frac{1}{\kappa} \vert c'(u_0)\vert^3\vert R_0\vert( \bar x) $, 
\end{itemize}
then wave breaking occurs along the backward characteristic $y(t, \xi)$ with $y(0,\xi)=\bar x$ within the time interval  
 \begin{equation*}
[t_l, t_u]=[ \frac{3}{5\tilde b}, \frac{3}{b}] , 
\end{equation*}
where 
\begin{equation*}\
b= \frac{1}{4\kappa}\left(1- \frac{1}{2A}\right)\vert c'(u_0)R_0\vert (\bar x)< \frac14 \kappa \vert c'(u_0)R_0\vert (\bar x)= \tilde b .
\end{equation*}
\end{theorem}

\subsection{Wave breaking along forward characteristics}

Again, let $(\xi_0, \eta_0)\in \Real^2$, which corresponds to the point $(0, \bar x)= (t(\xi_0, \eta_0), x(\xi_0,\eta_0))$ in Eulerian coordinates, such that $x_\eta(\xi_0, \eta_0)>0$. Then wave breaking occurs along the backward characteristic $z(t, \eta_0)$ if the function $S(t, z(t, \eta_0))=(u_t-c(u)u_x)(t, z(t, \eta_0))$ becomes unbounded, which, as highlighted in Section~\ref{sec:back}, is equivalent to $x_\eta(\xi, \eta_0)$ tending to $0$ along the horizontal line $\{(\xi, \eta_0)\mid \xi \in \Real\}$ in the $(\xi, \eta)$ plane. In analogy to Section~\ref{sub:BC}, it is natural to ask the following question: Does there exist a point $(\bar \xi, \eta_0)$ close to $(\xi_0, \eta_0)$ such that $x_\eta(\xi, \eta_0)\to 0$ as $\xi\to \bar \xi?$ If so, wave breaking either occurred recently or will take place in the near future and it then remains to estimate the actual breaking time. 

To answer this question, we can follow the same line of reasoning as in Section~\ref{sub:BC}, since \eqref{eq:Lagr}, where the right hand side can be viewed as $F(U, Z_\xi, Z_\eta)$, satisfies 
\begin{equation*}
(Z_{\eta})_\xi= F(U, Z_\eta, Z_\xi)=  F(U, Z_\xi, Z_\eta)= (Z_\xi)_\eta.
\end{equation*}
However, some small and straightforward adjustments are needed, due to the sign difference in \eqref{rel:tx}. Thus, we here do not present all the details, but restrict ourselves to sketching the red line. 

Define 
\begin{equation*}
\check f(\xi,\eta_0)= \sqrt{\frac{c(U(\xi_0, \eta_0))}{c(U(\xi, \eta_0))}}  x_\eta(\xi, \eta_0),
\end{equation*}
which satisfies that $\check f(\xi, \eta_0) \to 0$ as $\xi\to \bar \xi$ if and only if $x_\eta(\xi, \eta_0) \to 0$ as $\xi\to \bar \xi$. Furthermore, 
\begin{equation*}
\check f(\xi, \eta_0)= \check f(\xi_0, \eta_0) e^{\int_{\xi_0}^\xi \frac{c'(U)}{c(U)}\frac{U_\eta}{x_\eta} x_\xi (\tilde \xi, \eta_0) d\tilde \xi},
\end{equation*}
and hence our goal is to establish a criterion, which guarantees that 
\begin{equation}\label{asymp2}
\int_{\xi_0}^\xi \frac{c'(U)}{c(U)}\frac{U_\eta}{x_\eta} x_\xi (\tilde \xi, \eta_0) d\tilde \xi \to -\infty \quad \text{ as } \quad \xi \to \bar \xi.
\end{equation}

The function 
\begin{equation*}
\check h(\xi, \eta_0)= c(U)\frac{U_\eta}{x_\eta}(\xi, \eta_0),
\end{equation*}
which plays the same role as $h(\xi_0, \eta)$ in Section~\ref{sub:BC}, satisfies the differential equation
\begin{equation*}
\check h_\xi(\xi, \eta_0)= \check \alpha (\xi, \eta_0)+ \check \gamma \check h^2(\xi, \eta_0),
\end{equation*}
where 
\begin{equation*}
\check \alpha (\xi, \eta_0)= \frac{c'(U)}{2c^2(U)}J_\xi (\xi, \eta_0)
\end{equation*}
and 
\begin{equation*}
\check \gamma (\xi, \eta_0)= -\frac{c'(U)}{4c^2(U)} x_\xi (\xi, \eta_0). 
\end{equation*}

Assume that $\xi_0\leq \xi$, which implies, using \eqref{signt}, that $t(\xi_0, \eta_0)\leq t(\xi, \eta_0)$. Then \eqref{asymp2} requires that 
\begin{equation*}
c'(U) \check h(\xi, \eta_0) \to -\infty \quad \text{ as } \quad \xi\to \bar \xi
\end{equation*}
for some $\xi_0\leq \bar \xi$. Following the same lines as in the case $\eta\geq \eta_0$ in Section~\ref{sub:BC} and recalling \eqref{rel:tx} we obtain that if there exists a constant $A\geq 1$ such that 
\begin{equation*}
\vert \check h(\xi_0, \eta_0)\vert = A\lambda \kappa^2 (\mu_0+\nu_0)(\Real), 
\end{equation*}
then 
\begin{align*}
2& \ln\left(1+ \check h(\xi_0, \eta_0) \int_ {\xi_0}^\xi \frac14 \frac{c'(U)}{c(U)}t_\xi (\tilde \xi, \eta_0) d\tilde \xi\right)
\leq \int_{\xi_0}^\xi \frac{c'(U)}{2c(U)} \frac{U_\eta}{x_\eta} x_\xi (\tilde \xi, \eta_0) d\tilde \xi\\
&\qquad \qquad \qquad \qquad \qquad  \leq 2 \ln\left( 1+ \left(1- \frac1{2A}\right)\check h(\xi_0, \eta_0)\int_{\xi_0}^\xi \frac14 \frac{c'(U)}{c(U)} t_\xi(\tilde \xi, \eta_0) d\tilde \xi \right).
\end{align*}
Since $t_\xi(\xi, \eta)$ has the opposite sign as $t_\eta(\xi, \eta)$, the analogue to the above estimates in Section~\ref{sub:BC} are given by \eqref{est:bound3} and \eqref{est:bound4}, where $\eta\leq \eta_0$. The main difference from here on is that 
\begin{equation}\label{h:R}
c'(U)h(\xi_0, \eta_0)= c'(u_0)(u_{0,t}+c(u_0)u_{0,x})( \bar x)= c'(u_0)R_0(\bar x) 
\end{equation}
while 
\begin{equation}\label{h:S}
c'(U)\check h(\xi_0, \eta_0)= -c'(u_0) (u_{0,t}-c(u_0)u_{0,x})( \bar x)=-c'(u_0)S_0(\bar x),
\end{equation}
and following the proof of Theorem~\ref{thm:wb2} we end up with the following theorem.

\begin{theorem}[Wave breaking along forward characteristics - Part 1]\label{thm:wb3}
Given some initial data $(u_0, R_0, S_0, \mu_0, \nu_0)\in \mathcal{D}$, denote by $(u(t), R(t), S(t), \mu(t), \nu(t))$ the global conservative solution to the NVW equation at time $t$ and by $D$ the constant from Lemma~\ref{lem:hol}. If the following conditions are satisfied
\begin{itemize}
\item  $c'(u_0)S_0(\bar x)=c'(u_0)(u_{0,t}-c(u_0)u_{0,x})(\bar x)>0$, 
\item there exists $A\geq 1$ such that 
\begin{equation*}
  \vert S_0\vert ( \bar x)= A \lambda \kappa^2 (\mu_0+ \nu_0)(\Real), \quad \text{ and }
\end{equation*}
\item $
12 \bar \lambda^2D^2(1+\kappa) <(1-\frac{1}{2A})\frac{1}{\kappa} \vert c'(u_0)\vert ^3\vert S_0\vert(\bar x) $,
\end{itemize}
then wave breaking occurs along the forward characteristic $z(t, \eta)$ with $z(0,\eta)=\bar x$ within the time interval 
 \begin{equation*}
[t_l, t_u]=[ \frac{3}{5\tilde b}, \frac{3}{b}] , 
\end{equation*}
where 
\begin{equation*}
b= \frac{1}{4\kappa}\left (1- \frac{1}{2A}\right)\vert c'(u_0)S_0\vert (\bar x)< \frac14 \kappa \vert c'(u_0)S_0\vert (\bar x)= \tilde b .
\end{equation*}

\end{theorem}

Finally, assume that $\xi\leq \xi_0$, which implies, using \eqref{signt}, that $t(\xi_0, \eta_0)\geq t(\xi, \eta_0)$. Then \eqref{asymp2} requires that 
\begin{equation*}
c'(U) \check h(\xi, \eta_0) \to \infty \quad \text{ as } \quad \xi\to \bar \xi
\end{equation*}
for some $\bar \xi\leq \xi_0$. Following the same lines as in the case $\eta\leq \eta_0$ in Section~\ref{sub:BC} and recalling \eqref{rel:tx} we obtain that if there exists a constant $A\geq 1$ such that 
\begin{equation*}
\vert \check h(\xi_0, \eta_0)\vert = A\lambda \kappa^2 (\mu_0+\nu_0)(\Real), 
\end{equation*}
then 
\begin{align*}
2& \ln\left(1- \check h(\xi_0, \eta_0) \int_ {\xi}^{\xi_0} \frac14 \frac{c'(U)}{c(U)}t_\xi (\tilde \xi, \eta_0) d\tilde \xi\right)
\leq \int_{\xi_0}^\xi \frac{c'(U)}{2c(U)} \frac{U_\eta}{x_\eta} x_\xi (\tilde \xi, \eta_0) d\tilde \xi\\
&\qquad \qquad \qquad \qquad \qquad  \leq 2 \ln\left( 1- \left(1- \frac1{2A}\right)\check h(\xi_0, \eta_0)\int_{\xi}^{\xi_0} \frac14 \frac{c'(U)}{c(U)} t_\xi(\tilde \xi, \eta_0) d\tilde \xi \right).
\end{align*}
Since $t_\xi(\xi, \eta)$ has the opposite sign as $t_\eta(\xi, \eta)$, the analogue to the above estimates in Section~\ref{sub:BC} are given by \eqref{est:bound1} and \eqref{est:bound33}, where $\eta\geq \eta_0$. Recalling \eqref{h:R} and \eqref{h:S} and following the proof of Theorem~\ref{thm:wb1}, we obtain the following result.

\begin{theorem}[Wave breaking along forward characteristics - Part 2]\label{thm:wb4}
Given some initial data $(u_0, R_0, S_0, \mu_0, \nu_0)\in \mathcal{D}$, denote by $(u(t), R(t), S(t), \mu(t), \nu(t))$ the global conservative solution to the NVW equation at time $t$ and by $D$ the constant from Lemma~\ref{lem:hol}. If the following conditions are satisfied
\begin{itemize}
\item  $c'(u_0)S_0(\bar x)=c'(u_0)(u_{0,t}-c(u_0)u_{0,x})(\bar x)<0$, 
\item there exists $A\geq 1$ such that 
\begin{equation*}
  \vert S_0\vert ( \bar x)= A \lambda \kappa^2 (\mu_0+ \nu_0)(\Real) \quad \text{ and }
\end{equation*}
\item 
$
12 \bar \lambda^2D^2(1+\kappa) <(1-\frac{1}{2A})\frac{1}{\kappa} \vert c'(u_0)\vert ^3\vert S_0\vert( \bar x)  $,
\end{itemize}
then wave breaking occurred along the forward characteristic $z(t, \eta)$ with $z(0,\eta)=\bar x$ within the time interval 
 \begin{equation*}
[t_l, t_u]=[-\frac{3}{b}, - \frac{3}{5\tilde b}] , 
\end{equation*}
where 
\begin{equation*}\
b= \frac{1}{4\kappa}\left (1- \frac{1}{2A}\right )\vert c'(u_0)S_0\vert (\bar x)< \frac14 \kappa \vert c'(u_0)S_0\vert (\bar x)= \tilde b .
\end{equation*}
\end{theorem}

\subsection{Remark}\label{sec:Rem1} All the derived criteria for wave breaking have in common that either $R= u_t+c(u)u_x$ or $S=u_t-c(u)u_x$ become unbounded, which implies that either $u_t$ or $u_x$ or both become unbounded. However, it can be shown, and this will be done next, that wave breaking can only occur if $c'(u)u_x$ becomes unbounded, which rules out the possibility that $u_x$ remains bounded while $u_t$ blows up due to \eqref{cond:cder}. 

\subsubsection{Backward characteristics} When wave breaking occurs along a backward characteristic, which corresponds to the vertical line $\{(\xi_0, \eta) \mid \eta\in \Real\}$ in the $(\xi, \eta)$ plane, the integral 
\begin{equation}\label{eq:int}
  \int_{\eta_0}^\eta \frac{c'(U)}{2c(U)}\frac{U_\xi}{x_\xi}x_\eta(\xi_0, \tilde\eta) d\tilde \eta,
\end{equation}
tends to $-\infty$. Since \begin{equation*}
  U(\xi,\eta)= u(t(\xi, \eta), x(\xi, \eta)),
\end{equation*}
implies 
\begin{align}\nonumber
  U_\xi(\xi,\eta)&= u_t(t(\xi, \eta), x(\xi, \eta))t_\xi(\xi, \eta)+ u_x(t(\xi, \eta), x(\xi, \eta))x_\xi(\xi, \eta)\\ \label{rel:Uxi}
  & = \frac{1}{c(U(\xi,\eta))}(u_t+c(u)u_x)(t(\xi, \eta), x(\xi, \eta))x_\xi(\xi, \eta)
\end{align}
and 
\begin{align}\nonumber
  U_\eta(\xi, \eta)& = u_t(t(\xi, \eta), x(\xi, \eta))t_\eta(\xi, \eta)+ u_x(t(\xi, \eta), x(\xi, \eta))x_\eta(\xi, \eta)\\ \label{rel:Ueta}
  & = - \frac{1}{c(U(\xi, \eta))}(u_t-c(u)u_x)(t(\xi, \eta), x(\xi, \eta)) x_\eta(\xi, \eta),
\end{align}
the following identity holds
\begin{align*}
  \frac{U_\xi}{x_\xi}x_\eta(\xi, \eta)& = \frac{1}{c(U(\xi, \eta))}(u_t+c(u)u_x)(t(\xi, \eta), x(\xi, \eta))x_\eta(\xi, \eta)\\
  & = -U_\eta(\xi, \eta)+ 2u_x(t(\xi, \eta), x(\xi, \eta))x_\eta(\xi, \eta).
\end{align*}
Therefore the integral \eqref{eq:int} can be rewritten as 
\begin{align*}
  \int_{\eta_0}^\eta \frac{c'(U)}{2c(U)}\frac{U_\xi}{x_\xi}x_\eta(\xi_0, \tilde\eta) d\tilde \eta& = - \int_{\eta_0}^ \eta \frac{c'(U)}{2c(U)}U_\eta(\xi_0, \tilde \eta) d\tilde \eta \\
  & \quad + \int_{\eta_0}^\eta \frac{c'(u)}{c(u)} u_x(t(\xi_0, \tilde \eta), x(\xi_0, \tilde \eta)) x_\eta(\xi_0, \tilde \eta)d\tilde \eta\\
  & = -\frac12 ( \ln(c(U(\xi_0, \eta)))- \ln (c(U(\xi_0, \eta_0))))\\
  & \quad +  \int_{\eta_0}^\eta \frac{c'(u)}{c(u)} u_x(t(\xi_0, \tilde \eta), x(\xi_0, \tilde \eta)) x_\eta(\xi_0, \tilde \eta) d\tilde \eta,
\end{align*}
where the first term on the right hand side is always finite due to \eqref{cond:c}. As a consequence the integral on the left hand side can only tend to $-\infty$ if and only if the second integral on the right hand side tends to $-\infty$. Moreover, \eqref{rel:tx} and \eqref{def:crp} imply that 
\begin{align*}
  \int_{\eta_0}^\eta \frac{c'(u)}{c(u)}  u_x(t(\xi_0, \tilde \eta), x(\xi_0, \tilde \eta))x_\eta (\xi_0, \tilde \eta)d \tilde \eta &= - \int_{\eta_0}^\eta c'(u)u_x(t(\xi_0, \tilde \eta), x(\xi_0, \tilde \eta))t_\eta(\xi_0, \tilde\eta) d\tilde \eta\\
  & = - \int_{t(\xi_0, \tilde \eta_0)}^{t(\xi_0, \eta)} c'(u)u_x(l, y(l, \xi_0)) dl,
\end{align*} 
and hence wave breaking can only occur in the future if $c'(u)u_x(l, y(l, \xi_0))$ tends to $\infty$ and in the past if $c'(u)u_x(l, y(l, \xi_0))$ tends to $- \infty$.

\subsubsection{Forward characteristics} When wave breaking occurs along a forward characteristic, which corresponds to the horizontal line $\{ (\xi, \eta_0)\mid \xi\in \Real\}$ in the $(\xi, \eta)$ plane, we can follow the same lines to obtain that 
\begin{equation}\label{eq:int2}
  \int_{\xi_0}^\xi \frac{c'(U)}{2c(U)}\frac{U_\eta}{x_\eta} x_\xi(\tilde\xi, \eta_0) d\tilde \xi \to -\infty
\end{equation}
if and only if 
\begin{equation}\label{eq:cond}
 \int_{t(\xi_0, \eta_0)}^{t(\xi, \eta_0)} c'(u)u_x(l, z(l, \eta_0)) dl \to  - \infty .
 \end{equation}
Thus wave breaking can only occur in the future if $c'(u)u_x(l, z(l, \eta_0))$ tends to $-\infty$ and in the past if $c'(u)u_x(l, z(l, \eta_0))$ tends to $\infty$. 

Although these observations give more insight into the behavior of solutions close to wave breaking, they are not well suited for predicting wave breaking. If we only know that $c'(u)u_x(0, \bar x)$ is very large, we cannot distinguish between possible wave breaking in the future along a backward characteristic and in the past along a forward characteristic. This is also the reason why our wave breaking criteria are formulated in terms of $c'(u)R= c'(u)(u_t+c(u)u_x)$ and $c'(u)S=c'(u)(u_t-c(u)u_x)$. 

\subsection{Remark}\label{sec:Rem2}
All the derived criteria for wave breaking have in common that either $R= u_t+c(u)u_x$ or $S=u_t-c(u)u_x$ become unbounded. However, it can be shown, and this will be done next, that if wave breaking occurs at a point $(\bar t,\bar x)= (\bar t , y(\bar t, \xi_0))= (\bar t, z(\bar t, \eta_0))$, where $c'(u(\bar t, \bar x))\not =0$ and only $R$ or $S$, but not both become unbounded, which can be phrased in Lagrangian coordinates as either $x_\xi(\xi_0, \eta_0)=0$ or $x_\eta(\xi_0, \eta_0)=0$, but not both, then 
\begin{equation}\label{asymp:R}
R(t, y(t, \xi))\to \pm \infty \text{ as } t \to \bar t-  \quad \text{ and }\quad  R(t, y(t,\xi))\to \mp \infty \text{ as } t \to \bar t+
\end{equation}
or 
\begin{equation}\label{asymp:S}
S(t, z(t, \eta))\to \pm \infty \text{ as } t \to \bar t-  \quad \text{ and }\quad  S(t, z(t,\eta))\to \mp \infty \text{ as } t \to \bar t+
\end{equation}
where $t\to \bar t\mp$ denotes the left and right limit, respectively. 

\subsubsection{Backward characteristics} When wave breaking occurs along a backward characteristic, which corresponds to the vertical line $\{(\xi_0, \eta)\mid \eta\in \Real\}$ in the $(\xi, \eta)$ plane, we have that 
\begin{equation*}
x_\xi(\xi_0, \eta) \to 0 \quad \text{ and } \quad U_\xi(\xi_0, \eta) \to 0 \quad \text{ as } \eta\to  \eta_0
\end{equation*} 
for some $ \eta_0\in \Real$. 
Furthermore, we have by \eqref{eq:Lagr} that 
\begin{equation}\label{time:Uxi}
(\sqrt{c(U)}U_\xi)_\eta(\xi_0,  \eta_0)= \frac{c'(U)}{2c^{5/2}(U)}(x_\xi J_\eta+ x_\eta J_\xi)(\xi_0, \eta_0),
\end{equation}
and, since by assumption $c'(U(\xi_0, \eta_0))\not =0$ and $x_\eta(\xi_0, \eta_0)\not =0$, \eqref{prop:rel} implies that 
\begin{equation*}
(\sqrt{c(U)}U_\xi)_\eta(\xi_0, \eta_0)\not= 0.
\end{equation*}
Or, in other words, due to \eqref{cond:c}, $U_\xi(\xi_0, \cdot)$ changes sign at $\eta= \eta_0$ and there exists $\varepsilon>0$ such that $x_\xi(\xi_0, \eta)>0$ for all $\eta\in (\eta_0-\varepsilon, \eta_0+\varepsilon)\backslash \eta_0$. In view of \eqref{rel:Uxi}, this observation implies \eqref{asymp:R}.

\subsubsection{Forward characteristics}
When wave breaking occurs along a forward characteristic, which corresponds to the horizontal line $\{(\xi, \eta_0)\mid \xi \in \Real\}$ in the $(\xi, \eta)$ plane, we have that 
\begin{equation*}
x_\eta (\xi, \eta_0)\to 0 \quad \text{ and } \quad U_\eta(\xi, \eta_0)\to 0 \quad \text{ as } \quad \xi\to \xi_0
\end{equation*}
for some $\xi_0\in \Real$. Observing now that 
\begin{equation}\label{time:Ueta}
(\sqrt{c(U)}U_\eta)_\xi (\xi_0, \eta_0)= \frac{c'(U)}{2c^{5/2}(U)}(x_\xi J_\eta+ x_\eta J_\xi)(\xi_0, \eta_0)
\end{equation} 
and recalling that by assumption $c'(U(\xi_0, \eta_0))\not =0$ and $x_\xi(\xi_0, \eta_0)\not =0$, we can follow the same line of reasoning as in the case of wave breaking along backward characteristics, to conclude that \eqref{asymp:S} holds. 

\section{Examples}
We will now provide two examples to illustrate our results. First, we will apply our criteria for wave breaking to one of the few weak solutions of the NVW equation for nonconstant $c(u)$, which can be computed explicitly. This solution has been derived under the traveling wave ansatz, as exemplified in \cite{GlaHunZhe1996,AurKol2017,GruRei2021}. Unfortunately, as we shall demonstrate applying our wave breaking criteria, this solution is not a conservative solution. This is because the construction of traveling wave solutions in \cite{GruRei2021} only takes into account \eqref{NVW}, but not the time evolution of the measures $\mu$ and $\nu$, which is given by $J_{\xi, \eta}$ in Lagrangian coordinates.

Second, we consider some initial data, where the measures $\mu_0$ and $\nu_0$ are pure Dirac measures, and in addition, there exists some interval $[\alpha, \beta]\subset \Real$ such that $(u_0, R_0, S_0)= (\gamma, 0,0)$. If $c'(\gamma)=0$, then it turns out that the corresponding conservative solution satisfies $(u,R,S)\equiv (\gamma, 0,0)$ on the set $M=\{(t,x)\mid -\alpha+c(\gamma) \vert t\vert \leq x\leq \alpha- c(\gamma) \vert t\vert\}$. In addition, the Dirac measures $\mu$ and $\nu$ restricted to $M$, move in opposite directions with speed $c(\gamma)$ and hence behave locally, i.e., inside $M$, exactly like solutions of the linear wave equation.

\subsection{A traveling wave solution, that is not a conservative solution} Let us consider a traveling wave solution of \eqref{NVW}, $u(t,x)=w(x-st)$, where $w$ is a continuous and piecewise smooth function. 

As in 
 \cite{GruRei2021}, by inserting this ansatz into \eqref{NVW}, one finds
  \begin{equation*}
    s^2 w''(\zeta) = c(w)^2w''(\zeta) + c(w)c'(w)(w')^2(\zeta),
  \end{equation*}
  which can be rewritten as
  \begin{equation*}
    [(s^2-c(w(\zeta))^2)(w'(\zeta))^2]' = 0.
  \end{equation*}
Integrating once yields 
\begin{equation}\label{eq:tw:const}
    \abs{s^2 -c(w(\zeta))^2}(w')^2(\zeta) = k^2,
  \end{equation}
  where we have chosen the positive constant on the right to be the square of some $k \in \Real$.
  This implies that any smooth wave profile $w$ must satisfy
  \begin{equation}\label{eq:tw:slope}
    \sqrt{\abs{s^2-c(w(\zeta))^2}}w'(\zeta) = k
  \end{equation}
  for some constant $k$ wherever it is differentiable. Unfortunately, the resulting profiles $w$ will be monotone and hence do not belong to $H^1(\Real)$, which we require in Section~\ref{sec:back}. On the other hand, according to  \cite[Theorem 1.1]{GruRei2021} we can construct a piecewise smooth traveling wave solution $u(t,x)= w(x-st)$ for \eqref{NVW} by gluing together curve segments, which satisfy \eqref{eq:tw:slope} for different choices of $k$, at points where $\abs{s} = c(w)$.
  
  The traveling wave solution, we want to have a closer look at, is taken from \cite{GruRei2021} and we will refer to it as a hut. Let the wave speed be $c(u) = \sqrt{s^2 + \sin(u)}$ for $s>1$, and note that $c(u) = s$ for any $u$ such that $\sin(u) = 0$, and especially $u=0$ and $u=\pi$. Moreover, note that $u\equiv 0$ is a solution of \eqref{NVW}. We can therefore pick $\bar k>0$ and connect the constant state $w=0$, with an increasing wave profile $w$, which satisfies \eqref{eq:tw:slope} with $k= \bar k$ and connects $0$ with $\pi$, followed by a decreasing wave profile $w$, which satisfies \eqref{eq:tw:slope} with $k= -\bar k$ and connects $\pi$ with $0$, and then prolong $w$ by zero, as illustrated in Figure~\ref{fig:cuspedTW}.  
    \begin{figure}
    \includegraphics[width=0.8\linewidth]{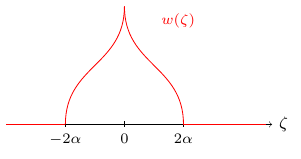}
    \caption{The ``hut'' traveling wave profile $w(\zeta)$.}
    \label{fig:cuspedTW}
  \end{figure}
  If we assume without loss of generality that the gluing point $\bar \zeta$, where $w(\bar\zeta)=\pi$, is located at $\bar \zeta=0$, then there exists some number $\alpha>0$, such that $w(-2\alpha)=w(2\alpha)=0$ and 
  \begin{equation*}
  w'(\zeta)=\begin{cases}
  \frac{\bar k}{\sqrt{\sin(w(\zeta))}}, & \zeta \in (-2\alpha,0),\\
  -\frac{\bar k}{\sqrt{\sin(w(\zeta))}}, & \zeta \in (0,2\alpha),\\
  0, & \text{otherwise.}
  \end{cases}
  \end{equation*}
  Thus $w$ belongs to $H^1(\Real)$ and according to  \cite[Theorem 1.1]{GruRei2021}  $u(t,x) = w(x-st)$ is a weak solution to \eqref{NVW}, which satisfies $u(t, \cdot)\in H^1(\Real)$ and $u_t(t, \cdot)\in L^2(\Real)$ for all $t\in \Real$. Furthermore, the functions $R(t,\cdot)$ and $S(t, \cdot)$, which are given by    \begin{equation*}
  R(t,x)= (u_t+c(u)u_x)(t,x)=-(s-c(w))w'(x-st) 
 \end{equation*}
 and 
 \begin{equation*}
  S(t,x)=(u_t-c(u)u_x)(t,x)= -(s+c(w))w'(x-st),
  \end{equation*}
   belong to $L^2(\Real)$ for all $t\in \Real$ and the associated measures $\mu(t)$ and $\nu(t)$ are purely absolutely continuous and given by \eqref{meas:abs}. In addition, and this sparks our interest in this  example, the slope of the nonconstant curve segments at the gluing points is unbounded. That means for the functions 
  \begin{equation}\label{rep:R}
  	R(t,x) = (c(w(x-st))-s) \frac{\pm \bar k }{\sqrt{c(w(x-st))^2-s^2} }= \pm \bar k \sqrt{\frac{c(w(x-st))-s}{c(w(x-st))+s}}
  \end{equation}
  and
  \begin{equation}\label{rep:S}
  	S(t,x)= -(c(w(x+st))+s) \frac{\pm \bar k }{\sqrt{c(w(x+st))^2-s^2}} = \mp \bar k  \sqrt{\frac{c(w(x+st))+s}{c(w(x+st))-s}},
  \end{equation}
since $s>1$ and $c(u)= \sqrt{s^2+\sin(u)}$, that $R(t,x)$ is bounded for all $(t,x)\in \mathbb{R}^2$, while $S^2(t,x)\to \infty$ as $(t,x)$ tends to $(t, \pm 2\alpha+st)$ or $(t, st)$. Or, in other words, wave breaking occurs for the solution $u(t,x)= w(x-st)$ at each time $t$ at the following 3 points: $-2\alpha+st$, $st$, and $2\alpha+st$. 
 
As stated earlier we claim that the hut traveling wave solution is a weak solution to \eqref{NVW}, but not a conservative solution in the sense of Section~\ref{sec:back}. One way of arguing, which we will follow, is to apply the wave breaking criteria from Theorem~\ref{thm:wb3} and Theorem~\ref{thm:wb4}, which allow us to predict locally the sign of $\bar u_t-c(\bar u)\bar u_x$ for the conservative solution $\bar u(t,x)$, which does not coincide with the one of the traveling wave solution $u(t,x)=w(x-st)$. 

If $u(t,x)= w(x-st)$ would be a conservative solution to \eqref{NVW}, then the corresponding initial data in $\D$ is given by 
\begin{equation*}
(u(0,x), R(0,x), S(0,x), \mu(0), \nu(0))= (w(x), R(0,x), S(0,x), \frac14 R^2(0,x), \frac14 S^2(0,x)),
\end{equation*}
where $R(0,x)$ and $S(0,x)$ are given by \eqref{rep:R} and \eqref{rep:S}, respectively. Thus it suffices to show that the conservative solution to the nonlinear variational wave equation $(\bar u(t,x), \bar R(t,x), \bar S(t,x), \bar \mu(t), \bar \nu(t))$
with initial data
\begin{align*}
(\bar u_0(x), \bar R_0(x),  \bar S_0(x), \bar \mu_0, \bar \nu_0)= (w(x), R(0,x), S(0,x), \frac14 R^2(0,x), \frac14 S^2(0,x)).
\end{align*}
satisfies $\bar u(\bar t,x)\not \equiv u(\bar t,x)$ for some $\bar t>0$.  

The function $c'(u_0)S_0(x)$ is continuous and strictly positive on the interval $(\alpha, 2\alpha)$. In addition, one has that 
\begin{equation*}
c'(u_0)S_0(x)\to 0 \quad \text{as }x\to \alpha \quad \text{ and } \quad  c'(u_0)S_0(x)\to \infty \quad \text{ as } x\to 2\alpha-.
\end{equation*}
Thus, Theorem~\ref{thm:wb3} implies, that there exists $\beta \in (\alpha, 2\alpha)$ such that wave breaking will occur along every forward characteristic starting in $(\beta, 2\alpha)$ in the near future. Indeed, let $I=\{x\in (\alpha, 2\alpha)\mid 1<2\cos(w(x))\}$, then 
\begin{equation*}
c'(w(x))=\frac{\cos(w(x))}{2c(w(x))}>\frac1{4\kappa} \quad \text{ for all } x\in I
\end{equation*}
and to each $A\geq 1$ there exists an interval $J_A= (\alpha_A, 2\alpha)\subset I$ such that for all $x\in J_A$
\begin{equation*}
\vert S_0(x)\vert \geq A\lambda \kappa^2 (\mu_0+\nu_0)(\Real) 
\end{equation*}
and 
\begin{equation*}
12 \bar \lambda^2 D^2 (1+\kappa) <\left(1-\frac{1}{2A}\right)\frac{1}{64\kappa^4} \vert S_0(x)\vert \leq\left (1- \frac{1}{2A}\right) \frac{1}{\kappa} \vert c'(u_0)\vert ^3 \vert S_0\vert (x). 
\end{equation*}
That means that for any given $A\geq 1$ each $x\in J_A$ satisfies all the assumptions of Theorem~\ref{thm:wb3} and in particular, we have that wave breaking occurs along each forward characteristic $z(t, \eta)$ with $z(0, \eta)=x\in J_A$ within the time interval 
\begin{equation}\label{cond:T1}
[0, T_A]= [0, \frac{48\kappa^2 }{1- \frac{1}{2A}} \sup_{x\in J_A}\frac{1}{S_0(x)}].
\end{equation}
Furthermore, we have that for each $t>0$ and every forward characteristic $z(t,\eta)$ with $z(0, \eta)=x\in I$ that
\begin{equation*}
c'(\bar u(t, z(t, \eta)))\geq c'(\bar u(0, z(0, \eta)))- \bar \lambda D\sqrt{1+\kappa} \sqrt{t}\geq \frac1{4\kappa} - \bar \lambda D\sqrt{1+ \kappa}\sqrt{t}
\end{equation*}
and in particular
\begin{equation}\label{cond:T2}
c'(\bar u(t, z(t, \eta)))>0 \text{ for all }t \in [0, 2T] \text{ if } 
T<\frac{1}{32 \bar \lambda^2D^2 \kappa^2(1+\kappa)}.
\end{equation}
Thus, choosing $A\geq 1$ big enough, we cannot only achieve that wave breaking occurs within the interval $[0,T_A]$ along each forward characteristic $z(t, \eta)$ with $z(0, \eta)\in J_A=(\alpha_A, 2\alpha)$, but also that $c'(\bar u(t, z(t, \eta)))>0$ for all $t \in [0, 2T_A]$. 

To summarize, we have shown so far that there exists $\bar A\geq 1$ minimal such that for the interval $J_{\bar A}=(\alpha_{\bar A}, 2\alpha)\subset (\alpha, 2\alpha)=I$ the following holds. For each forward characteristic $z(t, \eta)$ with $z(0, \eta)\in J_{\bar A}$
\begin{itemize}
\item wave breaking occurs within the time interval $[0, T_{\bar A}]$, where $T_{\bar A}$ is given by \eqref{cond:T1} and 
\item $c'(\bar u(t, z(t, \eta)))>0$ for all $t\in [0, 2T_{\bar A}]$.
\end{itemize}
This means, in view of Remark~\ref{sec:Rem1}, that to each $\eta$ which satisfies  $z(0, \eta)\in J_{\bar A}$, there exists a $\tau_\eta\in [0, T_{\bar A}]$ such that 
\begin{equation*}
\bar S(t, z(t, \eta))\to \infty \quad \text{ and }\quad  \bar u_x(t, z(t, \eta)) \to -\infty \quad \text{ as } t \to \tau_\eta-.
\end{equation*}

To determine the sign of $\bar S (t,z(t, \eta))$ for $t\in (\tau_\eta, 2T_{\bar A}]$, i.e., after wave breaking, for every $\eta$ such that $z(0, \eta)\in J_{\bar A}$, we want to take advantage of Remark~\ref{sec:Rem2}, which is applicable if we can show that no wave breaking occurs along the backward characteristics $y(t,\xi)$ with $y(0,\xi)\in (\beta,\infty)$ within the time interval $[0, 2T_{\bar A}]$.

Let $\bar x\in (\beta, \infty) \subset (\alpha, \infty)$, then $\bar R_0(\bar x)= (\bar u_{0,t}+c(\bar u_0)\bar u_{0,x})(\bar x)\leq 0$ and following the same lines as for the forward characteristics, we obtain for every backward characteristic $y(t,\xi)$ with $y(0,\xi)\in (\beta, \infty)$ that 
\begin{equation*}
c'(\bar u(t, y(t, \xi)))>0 \quad \text{ for all } t\in [0,2T_{\bar A}]. 
\end{equation*}
Furthermore, \eqref{rel:Uxi} implies that $U_\xi(\xi, \eta)\leq 0$ and $x_\xi(\xi, \eta)>0$ for all $(\xi, \eta)$ such that $(t(\xi, \eta), x(\xi, \eta))= (0, \bar x)$. Recalling \eqref{time:Uxi} and using that 
\begin{equation*}
\left(\frac{x_\xi}{\sqrt{c(U)}}  \right)_\eta(\xi, \eta)= \frac{c'(U)}{2c^{3/2}(U)} (x_\eta U_\xi)(\xi, \eta),
\end{equation*}
by \eqref{eq:Lagr2},
we find that for $\bar \eta\leq \eta$ 
\begin{equation}\label{last:ineq}
U_\xi(\xi,\bar\eta)\leq U_\xi(\xi, \eta)\leq 0 \quad \text{ and } \quad 0<x_\xi(\xi, \eta) \leq x_\xi(\xi, \bar \eta)
\end{equation}
as long as $c'(U(\xi, \bar\eta))$ does not change sign. However, we know that $y(t(\xi, \eta), \xi)=x(\xi, \eta)$ and hence, using \eqref{rel:Uxi} once more, we obtain for all $t\in [0, 2T_{\bar A}]$ that 
\begin{equation*}
-\infty <\bar R(t, y(t, \xi))\leq 0 \quad \text{ for all } \xi \text{ such that } y(0, \xi)\in (\beta, \infty)
\end{equation*}
and for every backward characteristic $y(t,\xi)$ with $y(0, \xi)=x\in (\beta, \infty)$ no wave breaking occurs within the time interval $[0, 2T_{\bar A}]$. Thus Remark~\ref{sec:Rem2} implies that 
\begin{equation*}
\bar S(t, z(t, \eta))\to - \infty \quad \text{ and } \quad \bar u_x(t, z(t, \eta)) \to \infty \quad \text{ as } t \to \tau_\eta+.
\end{equation*}
Moreover, by \eqref{time:Ueta}, one has that $\bar S(t, z(t, \eta)) <0 $ for all $t\in (\tau_\eta, 2T_J]$.

Last but not least, choose $\hat A\geq \bar A$, which implies that $J_{\hat A}\subset J_{\bar A}$ and $T_{\hat A}\leq T_{\bar A}$, such that $z(t, \eta)\in (st, \infty)$ for all $\eta$ with $z(0, \eta)\in J_{\hat A}$ and $t\in [0, 2T_{\hat A}]$. Then we have for $\bar t= \frac32 T_{\hat A}$ that 
\begin{equation*}
I_{\bar t}=\{x= z(\bar t, \eta)\mid z(0, \eta)\in J_{\hat A}\} \subset (s\bar t, \infty).
\end{equation*}
Furthermore, 
\begin{equation*}
\bar S(\bar t, x) <0 \quad \text{ while } \quad S(\bar t, x)\geq 0  \quad \text{ for all } x \in I_{\bar t},
\end{equation*}
and in particular $ \bar S (\bar t, \cdot)\not \equiv S(\bar t, \cdot)$. Therefore, the traveling wave solution $u(t,x)=w(x-st)$ cannot be a conservative solution.

\begin{remark}
As a closer look at \cite{GruRei2021} and \cite{HR} reveals, the main reason, why the above traveling wave solution is not a conservative solution, is that in \cite{GruRei2021} the main focus is on identifying all possible wave profiles $w$ such that $u(t,x)=w(x-st)$ is a weak solution to \eqref{NVW}. As a consequence, the equations describing the time evolution of the measures $\mu$ and $\nu$ are completely neglected therein. This is in contrast to \cite{HR}, where conservative solutions in $\D$ are constructed and to handle the concentration and the spreading of energy thereafter, it is essential to take the time evolution of these measures into account.  
\end{remark}

\subsection{A solution that behaves locally like the linear wave equation.} One property, which distinguishes the conservative solutions of the nonlinear variational wave equation from the conservative solutions of the Hunter--Saxton and the Camassa--Holm equation is the fact, that energy can concentrate on sets of measures zero, e.g., in the form of a Dirac delta, but the concentrated energy might not spread out immediately thereafter, as the following example illustrates.

Inspired by \cite[Footnote 4, p.~956]{HR}, we consider \eqref{NVW} for some initial data $(u_0, R_0, S_0, \mu_0, \nu_0)\in \D$, which satisfies the following assumption: There exists an interval $I=[-\alpha, \alpha]$ such that 
\begin{itemize}
\item $(u_0, R_0,S_0)(x)=(\gamma, 0, 0)$ for all $x\in I$,
\item $\mu_0\vert_I= a \delta_0$ and $\nu_0\vert _I= b \delta_0$ for some $a\geq 0$ and $b\geq 0$.
\end{itemize}
Furthermore, we require that $c\in C^2(\Real)$ and $c'(\gamma)=0$. 
Defining 
\begin{equation}\label{def:setM}
M=\{(t,x)\mid -\alpha+ c(\gamma)\vert t\vert \leq x\leq \alpha- c(\gamma)\vert t\vert\},
\end{equation}
we will show that
\begin{equation}\label{prop:E1}
(u,R,S)(t,x)= (\gamma, 0, 0) \quad \text{ for all } (t,x)\in M
\end{equation}
and for $I_t= [-\alpha+ c(\gamma)\vert t\vert, \alpha- c(\gamma)\vert t\vert]$
\begin{equation}\label{prop:E2}
\mu(t)\vert _{I_t}= a \delta_{-c(\gamma) t} \quad \text{ and } \quad \nu(t)\vert_{I_t}= b \delta_{c(\gamma)t}.
\end{equation}

\begin{figure}
   \includegraphics[width=0.7 \linewidth]{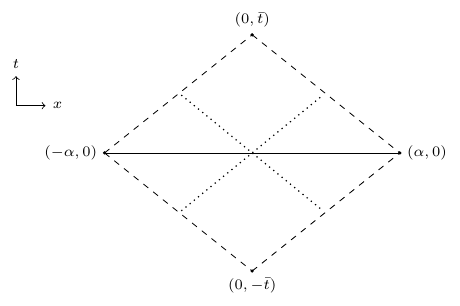}
    \includegraphics[width=0.7\linewidth]{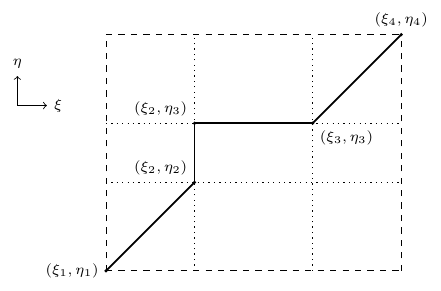}
    \caption{The set $M$ in Eulerian coordinates and the corresponding box $\Omega$ in Lagrangian coordinates. The thick curve in each plot shows the set on which the initial data is given. The dotted lines in the first picture show the position of the discrete part of $\mu(t)$ and $\nu(t)$ inside $M$, while the two stripes bounded by the dotted lines indicate the set where $x_\eta$ (vertically) and $x_\xi$ (horizontally) equal zero.}
    \label{fig:EulLag}
  \end{figure}

The set of Eulerian coordinates $\{(0,x) \mid x\in I\}$ corresponds to a curve $\bar{\mathcal{C}}_I\subset \bar{\mathcal{C}}$, which connects the point $(\xi_1, \eta_1)$ to the point $(\xi_4, \eta_4)$ as in Figure~\ref{fig:EulLag}. Along the vertical line which connects $(\xi_2, \eta_2)$ with $(\xi_2, \eta_3)$, we have 
\begin{equation*}
(t_\eta, x_\eta, U_\eta, J_\eta)= (0,0,0,1)
\end{equation*}
and $\eta_3-\eta_2=b$, while 
\begin{equation*}
(t_\xi, x_\xi, U_\xi, J_\xi)= (0,0,0,1)
\end{equation*}
along the horizontal line which connects $(\xi_2, \eta_3)$ with $(\xi_3, \eta_3)$ and $\xi_3-\xi_2=a$. Or in other words the vertical line segment corresponds to the discrete part of $\nu_0$, while the horizontal one represents the discrete part of $\mu_0$ inside $I$, which is a consequence of the measure $\mu_0+\nu_0$ being mapped to the function $J$. 

Let $\Omega= [\xi_1, \xi_4]\times[\eta_1, \eta_4]$. Then for any point $(\xi, \eta) \in \Omega$, we can fix one of the coordinates and find the corresponding point on $\bar{\mathcal{C}}_I$ with the same coordinate, which we denote by $(\xi, \mathcal{P}(\xi))$ and $(\mathcal{Q}(\eta), \eta)$, respectively. By assumption, we have $U(\xi, \eta)= \gamma$ for all $(\xi, \eta) \in \bar{\mathcal{C}}_I$, which can be rephrased as 
\begin{equation*}
U(\xi, \mathcal{P}(\xi))= \gamma= U(\mathcal{Q}(\eta), \eta) \quad \text{ for all } (\xi, \eta) \in \Omega.
\end{equation*}
Furthermore, 
\begin{equation*}
U_\xi(\xi, \mathcal{P}(\xi))=U_\xi(\mathcal{Q}(\eta), \eta) =0= U_\eta(\xi, \mathcal{P}(\xi))= U_\eta(\mathcal{Q}(\eta), \eta) \quad \text{ for all } (\xi, \eta) \in \bar{\mathcal{C}}_I,
\end{equation*}
since $R_0(x)=0=S_0(x)$ for all $x\in I$.
As we will see
\begin{equation}\label{cl:constu}
U(\xi, \eta)= \gamma \quad \text{ for all } (\xi, \eta) \in \Omega.
\end{equation}
Indeed, pick $(\xi, \eta) \in \Omega$ below $\bar{\mathcal{C}}_I$. Then $t(\xi, \eta) \geq 0$ and we can write 
\begin{align*}
U(\xi, \eta)& = U(\xi, \eta)- U(\xi, \mathcal{P}(\xi))+ \gamma\\
& = \gamma + \int_{\mathcal{P}(\xi)}^\eta \left(U_\eta (\xi, \tilde \eta)- U_\eta (\mathcal{Q}(\tilde \eta), \tilde \eta)\right ) d\tilde \eta\\
& = \gamma + \int_{\mathcal{P}(\xi)}^\eta \int_{\mathcal{Q}(\tilde \eta)}^\xi U_{\eta, \xi}(\tilde \xi, \tilde \eta) d\tilde \xi d\tilde \eta.
\end{align*}
Now, recalling \eqref{eq:Lagr3}, we find for all $(\xi, \eta) \in \Omega$,
\begin{equation*}
U_{\eta, \xi}(\xi, \eta)= \left(c'(U(\xi, \eta))- c'(\gamma)\right) \left( \frac{x_\xi J_\eta+x_\eta J_\xi}{2c^3(U)}- \frac{U_\xi U_\eta}{2c(U)}\right) (\xi, \eta),
\end{equation*}
where we have used $ c'(\gamma)=0$. Furthermore, \eqref{rel:seebreak} implies that 
\begin{align}\nonumber
(\sqrt{x_\xi J_\eta}- \sqrt{x_\eta J_\xi})^2 & \leq x_\xi J_\eta+ x_\eta J_\xi-c^2(U)U_\xi U_\eta \\ \label{CS}
& \leq (\sqrt{x_\xi J_\eta}+ \sqrt{x_\eta J_\xi})^2\leq 4 (x_\xi+J_\xi)(x_\eta+ J_\eta)
\end{align}
and hence, recalling \eqref{cond:cder} 
\begin{equation*}
\vert U_{\eta, \xi}(\xi, \eta) \vert \leq 2\bar \lambda\kappa^3 (x_\xi+J_\xi)(x_\eta+ J_\eta)(\xi, \eta) \vert U(\xi, \eta)- \gamma\vert.
\end{equation*}
Finally, choose $L>0$ such that $\Omega\subset S_L$, where $S_L$ is given by \eqref{def:SL}. Then there exist positive constants $C_1$ and $C_2$, such that \eqref{prop:rel} holds for all $(\xi, \eta) \in \Omega$ and 
\begin{equation*}
\vert U_{\eta, \xi }(\xi, \eta) \vert \leq  2\bar \lambda \kappa^3C_2^2\vert U(\xi, \eta) - \gamma \vert= B\vert U(\xi, \eta)- \gamma \vert.
\end{equation*}
Therefore we have 
\begin{equation*}
\vert U(\xi, \eta) - \gamma\vert \leq B \int_{\eta}^{\mathcal{P}(\xi)} \int_{\mathcal{Q}(\tilde \eta)}^{\xi} \vert U(\tilde \xi, \tilde \eta) - \gamma\vert d\tilde \xi d\tilde \eta. 
\end{equation*}
and introducing the function $\ell(\xi, \eta) = \vert U(\xi, \eta)- \gamma\vert $ the above inequality can be written as
\begin{equation*}
0 \leq \ell (\xi, \eta) \leq B\int_{\eta}^{\mathcal{P}(\xi)} \int_ {\mathcal{Q}(\tilde\eta)}^{\xi} l (\tilde \xi, \tilde \eta) d\tilde \xi d\tilde \eta=h(\xi, \eta),
\end{equation*}
where we also introduced the function $h(\xi, \eta)$. Differentiating $h(\xi, \eta)$, we obtain 
\begin{equation*}
h_\eta( \xi, \eta)= - B \int_{\mathcal{Q}(\eta)}^\xi l(\tilde \xi, \eta) d\tilde \xi \geq - B \int_{\mathcal{Q}(\eta)}^\xi h(\tilde \xi, \eta)d\tilde \xi \geq - B (\xi- \mathcal{Q}(\eta)) h(\xi, \eta),
\end{equation*}
where the last inequality follows since $h( \cdot, \eta)$ is increasing. Thus we can apply Gronwall's inequality to obtain 
\begin{equation*}
h(\xi, \eta) \leq \exp\left(B \int_{\eta}^{\mathcal{P}(\xi)}\int_{\mathcal{Q}(\tilde \eta)}^\xi d\tilde \xi d\tilde \eta\right) h( \xi , \mathcal{P}(\xi))=0, 
\end{equation*}
since $h(\xi, \mathcal{P}(\xi))=0$ for all $\xi \in [\xi_1, \xi_4]$, which implies that 
\begin{equation*}
0 \leq \vert U(\xi, \eta)- \gamma\vert =\ell(\xi, \eta) \leq h(\xi, \eta)=0.
\end{equation*}
Following the same line of reasoning with slight modifications for $(\xi, \eta)\in \Omega$ above $\bar{\mathcal{C}}_I$, we end up with \eqref{cl:constu}.

As an immediate consequence we obtain that 
\begin{equation}\label{sol:box1}
c(U(\xi, \eta))= c(\gamma) \quad \text{ and }\quad c'(U(\xi, \eta))= 0 \quad \text{ for all }(\xi, \eta) \in \Omega,
\end{equation}
which in turn implies using \eqref{eq:Lagr}
\begin{equation}\label{sol:box2}
(t_\xi, x_\xi, U_\xi, J_\xi)(\xi, \eta) = (t_\xi, x_\xi, U_\xi, J_\xi) (\xi, \mathcal{P}(\xi)) \quad \text{ for all } (\xi, \eta) \in \Omega
\end{equation}
and
\begin{equation}\label{sol:box3}
(t_\eta, x_\eta, U_\eta, J_\eta)(\xi, \eta)= (t_\eta, x_\eta, U_\eta, J_\eta) (\mathcal{Q}(\eta), \eta) \quad \text{ for all } (\xi, \eta) \in \Omega.
\end{equation}
From these observations it now follows that the set $M$ in Eulerian coordinates coincides with the set $\Omega$ in Lagrangian coordinates, since vertical lines in the $(\xi, \eta)$ plane correspond to backward characteristics in the $(t,x)$ plane, while horizontal lines correspond to forward characteristics. Thus, the mapping from Lagrangian to Eulerian coordinates implies \eqref{prop:E1}. 

To finally conclude that also \eqref{prop:E2} holds, observe that in the vertical stripe bounded by the dotted lines in the $(\xi, \eta)$ plane $J_\xi \equiv 1$ and $x_\xi\equiv 0$, while $J_\xi\equiv 0$ outside, which together with $\xi_3-\xi_2=a$ gives rise to the first equality in \eqref{prop:E2}. For the second one, we use that in the horizontal stripe bounded by the dotted lines in the $(\xi, \eta)$ plane $J_\eta\equiv 1$ and $x_\eta\equiv 0$, while $J_\eta\equiv 0$ outside, which together with $\eta_3-\eta_2=b$ finishes the proof. 

\begin{remark} 
Although we restricted ourselves to considering the case of a finite interval $I=[-\alpha, \alpha]$ and measures $\mu_0=a\delta_0$ and $\nu_0=b\delta_0$, the above result can be extended to the case of the whole real line, in which case one has to require $\gamma=0$, or to the case of $\mu_0$ and $\nu_0$ being a finite sum of Dirac masses. 
\end{remark}

Last but not least we want to take this observation as a starting point to show that outside the box $[\xi_1, \xi_4]\times [\eta_1, \eta_4]$ the behavior can be quite different. In particular, we focus on showing that the Dirac mass for $\mu$ must not exist for all times $t\in \Real$.

Consider the following initial data $(u_0, R_0, S_0, \mu_0, \nu_0)\in \D$: Let $0<\alpha <\beta$ such that 
\begin{itemize}
\item $(u_0, R_0, S_0)(x)= (\star, 0, 2)$ for all $x\in [-\beta , -\alpha)$,
\item $(u_0, R_0, S_0)(x)= (\gamma, 0, 0)$ for all $x\in [-\alpha, \alpha]$,
\item $(u_0, R_0, S_0)(x)= (\star, 0, -2)$ for all $x\in (\alpha, \beta]$
\item $(u_0, R_0, S_0)(x)= (0,0,0)$ for all $x\in \Real \backslash[-\beta, \beta]$,
\item $\mu_0=a\delta_0$ and $\nu_0\vert= \mathbbm{1}_{[-\beta,-\alpha)\cup (\alpha, \beta]}+ b\delta_0$ for some $a\geq 0$ and $b\geq 0$.
\end{itemize}
Here $\star$ denotes the value of $u_0(x)$, which can be computed if the function $c(x)$ is known, since we have $c(u_0)u_{0,x}(x)= \frac12 (R_0-S_0)(x)=- \frac12 S_0(x)=\mp 1$ for $x\in [-\beta, -\alpha) \cup (\alpha, \beta]$ and $u_0(-\alpha)=\gamma = u_0(\alpha)>0$. Furthermore, choose $c(x)\in C^2(\Real)$ such that $c(\gamma)=1$, $c'(\gamma)=0$, $c(x)$ is strictly decreasing on $(-\infty, \gamma]$ and strictly increasing on $[\gamma, \infty)$. Then we have
\begin{equation}\label{para}
(\mu_0+ \nu_0)(\Real)=2(\beta-\alpha)+a+b
\end{equation}
and since $\beta-\alpha$ is dependent on $c(x)$ and hence $\gamma$, the total energy can be made arbitrarily small, by choosing $a$, $b$, and $\gamma$ close enough to zero. 

\begin{figure}
   \includegraphics[width=0.7 \linewidth]{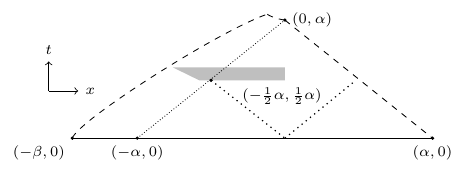}
    \includegraphics[width=0.7\linewidth]{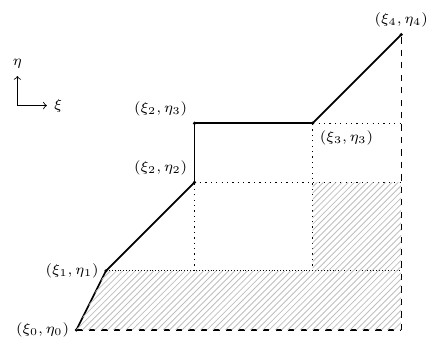}
    \caption{The set in Eulerian coordinates, which is bounded by the forward characteristic through $(0, -\beta)$ and the backward characteristic through $(0, \alpha)$ for $t\geq 0$ and the corresponding area in Lagrangian coordinates. The thick curve in each plot shows the set on which the initial data is given. The grey area in the picture in Eulerian coordinates indicates the region, where the measure $\mu$ is purely absolutely continuous while the dotted lines show the position of the discrete part of $\mu(t)$ and $\nu(t)$, respectively. In the second plot the shaded area represents the area in Lagrangian coordinates, which contains all the points in the grey area in Eulerian coordinates as a proper subset.}
    \label{fig:EulLag2}
  \end{figure}

As before, the set of Eulerian coordinates $\{(0,x)\mid x\in [-\beta, \alpha]\}$ corresponds to a curve $\bar{\mathcal{C}}_I \subset \bar{\mathcal{C}}$, which connects the point $(\xi_0, \eta_0)$ to the point $(\xi_4, \eta_4)$ as in Figure~\ref{fig:EulLag2}. Due to the choice of the initial data, the curve segment connecting $(\xi_1, \eta_1)$ with $(\xi_4, \eta_4)$ corresponds to $x\in [-\alpha, \alpha]$, which implies that the conservative solution in Lagrangian coordinates satisfies \eqref{sol:box1}-- \eqref{sol:box3} for $\Omega= [\xi_1, \xi_4]\times [\eta_1, \eta_4]$. In Eulerian coordinates we thus have that \eqref{def:setM}--\eqref{prop:E2} hold and $-c(\gamma)t=-t\in M$, where $M$ is given by \eqref{def:setM}, if $-\frac12\alpha \leq t\leq \frac12\alpha$. 
We will show that for suitable choices of $a$, $b$, $\alpha$, $\beta$, and $\gamma$, there exist $\varepsilon>0$ and $\delta>0$ such that for all $t \in (\frac12 \alpha, \frac12 \alpha+ \varepsilon)$ 
\begin{equation}\label{ref:claim}
\mu(t)\vert_{J_t}= \mu_{ac}(t)\vert_{J_t},
\end{equation}
where $J_t= [-\delta t, 0]$. Or in other words the energy which is concentrated at $(t,x)= (\frac12 \alpha,-\frac12 \alpha)$ does not remain concentrated forward in time, but is spread out immediately. 

To prove this result, we take advantage of the curve segment connecting $(\xi_0, \eta_0)$ with $(\xi_1, \eta_1)$, which corresponds to $x\in [-\beta, -\alpha]$ and satisfies $\eta= 2(\xi-\xi_0)+ \eta_0$. 
By assumption, we have along this line 
\begin{equation}\label{diag}
(x_\xi, c(U)U_\xi, J_\xi)=(\frac12, 0, 0) \quad \text{ and } \quad  (x_\eta, c(U)U_\eta, J_\eta)= ( \frac14, -\frac12, \frac12).
\end{equation} 

As a first step towards our claim, we want to show that $c(U)U_\eta(\xi, \eta) <0$ for all $(\xi, \eta) \in [\xi_0, \xi_4]\times [\eta_0, \eta_1]$, which lie below the curve $\bar{\mathcal{C}}_I$ and satisfy $t(\xi, \eta)\leq \frac34 \alpha$. 

Pick $(\xi, \eta)\in [\xi_0, \xi_4]\times [\eta_0,\eta_1]$ such that $0\leq t(\xi, \eta)\leq \frac34 \alpha$. Then $(\xi, \eta)$ lies below $\bar{\mathcal{C}}_I$ and \eqref{cond:c}, \eqref{cond:cder}, \eqref{rel:seebreak}, and \eqref{eq:Lagr} imply 
\begin{equation*}
\left\vert \left(\frac{2x_\xi+J_\xi}{\sqrt{c(U)}}\right)_\eta \right\vert \leq \frac14 \lambda \kappa^{3/2} \frac{2x_\xi+J_\xi}{\sqrt{c(U)}}\frac{2x_\eta+ J_\eta}{\sqrt{c(U)}}.
\end{equation*}
Furthermore, $(\xi, \mathcal{P}(\xi))\in \bar{\mathcal{C}}_I$ and by \eqref{prop:G}, \eqref{rel:tx}, \eqref{asymp:J1}, \eqref{asymp:J2}, and \eqref{cond:c}
\begin{align*}
\frac{2x_\xi+J_\xi}{\sqrt{c(U)}}(\xi, \eta)& \leq \frac1{\sqrt{c(U)}} e^{\frac14 \lambda \kappa^{3/2}\int_{\eta}^{P(\xi)} \left(-2\sqrt{c(U)}t_\eta+ \frac{J_\eta}{\sqrt{c(U)}}\right)(\xi, \tilde \eta) d\tilde \eta}\\
& \leq \kappa^{1/2} e^{\frac14 \lambda \kappa^2 (2t(\xi, \eta)+ (\mu_0+\nu_0)(\Real)) }.
\end{align*}
Following the same lines one obtains 
\begin{equation*}
\frac{2x_\eta+J_\eta}{\sqrt{c(U)}}(\xi, \eta)\leq  \kappa^{1/2} e^{\frac14 \lambda \kappa^2 (2t(\xi, \eta)+ (\mu_0+\nu_0)(\Real)) }.
\end{equation*}

Next, observe that 
\begin{equation*}
(c(U)U_\eta)_\xi= \frac12\frac{c'(U)}{c^2(U)}(x_\xi J_\eta+ x_\eta J_\xi)+\frac12 c'(U) U_\xi U_\eta
\end{equation*}
and hence we have for all $\tilde \xi \in [\mathcal{Q}(\eta), \xi]$, cf. \eqref{CS},
\begin{align*}
\vert (c(U)U_\eta)_\xi(\tilde \xi, \eta) \vert & \leq 2\lambda \kappa \frac{x_\xi+ J_\xi}{\sqrt{c(U)}}\frac{x_\eta+J_\eta}{\sqrt{c(U)}} \leq 2 \lambda \kappa^2 e^{\frac12 \lambda \kappa^2(2\beta+a+b)},
\end{align*}
where we used \eqref{CS}, \eqref{para} and that we require $0 \leq t(\xi, \eta) \leq \frac34 \alpha$.
As an immediate consequence we obtain
\begin{align*}
c(U)U_\eta(\xi, \eta)& = c(U)U_\eta (Q(\eta), \eta)+ \int_{Q(\eta)}^\xi  (c(U)U_\eta)_\xi (\tilde \xi , \eta) d\tilde \xi \\
& \leq -\frac12 + 2 \lambda \kappa^2 e^{\frac12 \lambda \kappa^2(2\beta+a+b)} (\xi_4-\xi_0).
\end{align*} 
Since the definition of $\bar{\mathcal{C}}_I$ implies 
\begin{align*}
\xi_4-\xi_0= a+\alpha+\beta,
\end{align*}
we end up with 
\begin{equation*}
c(U)U_\eta(\xi, \eta)\leq -\frac12 +2 \lambda \kappa^2 e^{\frac12 \lambda \kappa^2(2\beta+a+b)} (a+\alpha+\beta)<-\frac14 ,
\end{equation*}
if $a$, $\alpha$, and $\gamma$ are chosen small enough. Recalling \eqref{cond:c} and \eqref{rel:seebreak}, it thus follow that there exist $a$, $b$, $\alpha$ and $\gamma$, such that 
\begin{equation}\label{good:sign}
x_\eta(\xi, \eta)>0, \quad U_\eta(\xi, \eta)<0, \quad  \text{ and } \quad J_\eta(\xi, \eta)>0
\end{equation}
for all $(\xi, \eta) \in [\xi_0, \xi_4]\times [\eta_0, \eta_1]$ with $0 \leq t(\xi, \eta) \leq \frac34 \alpha$. 
Furthermore, one has 
\begin{equation*}
U(\xi, \eta)- \gamma \geq U(\xi, \eta)- U(\xi, \min(\mathcal{P}(\xi), \eta_1))= -\int_{\eta}^{\min( \mathcal{P}(\xi), \eta_1)} U_\eta (\xi, \tilde \eta) d\tilde \eta >0
\end{equation*}
for all $(\xi, \eta) \in [\xi_0, \xi_4]\times [\eta_0, \eta_1)$ with $0 \leq t(\xi, \eta) \leq \frac34 \alpha$, and hence $c'(U)(\xi, \eta)>0$ for all such $(\xi, \eta)$. 

To conclude the proof of our claim \eqref{ref:claim}. Observe that \eqref{eq:Lagr} and \eqref{rel:seebreak} imply
\begin{align*}
U_{\xi, \eta}&=   \frac12\frac{c'(U)}{c^3(U)}(x_\xi J_\eta+ x_\eta J_\xi) - \frac12 \frac{c'(U)}{c(U)} U_\xi U_\eta\\ 
& \geq  \frac{c'(U)}{2c^3(U)} (\sqrt{x_\xi J_\eta}- \sqrt{x_\eta J_\xi})^2 \geq 0
\end{align*}
which combined with \eqref{diag} and $c(U)U_\xi(\xi, \eta) = 0 $ for all $(\xi, \eta) \in \Omega$, yields 
\begin{equation*}
U_\xi(\xi, \eta) \leq 0 \quad \text{ for all } (\xi, \eta) \in [\xi_0, \xi_4]\times [\eta_0, \eta_1] \text{ with } 0 \leq t(\xi, \eta)\leq \frac34 \alpha.
\end{equation*}
Therefore, 
\begin{align*}
(c(U)U_\xi)_\eta& = \frac12\frac{c'(U)}{c^2(U)}(x_\xi J_\eta+ x_\eta J_\xi)+\frac12 c'(U) U_\xi U_\eta\\
& = \frac{c'(U)}{2c^2(U)} (\sqrt{x_\xi J_\eta}+ \sqrt{x_\eta J_\xi})^2\geq 0 ,
\end{align*}
and recalling \eqref{cond:c} and \eqref{good:sign} we end up with 
\begin{equation*}
x_\xi(\xi, \eta)>0, \quad U_\xi(\xi, \eta)<0, \quad  \text{ and } \quad J_\xi(\xi, \eta)>0
\end{equation*}
for all $(\xi, \eta) \in [\xi_0, \xi_4]\times [\eta_0, \eta_1)$ with $0 \leq t(\xi, \eta) \leq \frac34 \alpha$. 
Thus, choosing
\begin{equation*}
\varepsilon= \min(\frac18 \alpha, \frac12(t(\xi_2, \eta_0)- t(\xi_2, \eta_1))) =\min(\frac18 \alpha,  \frac12 (t(\xi_2, \eta_0)-\frac12 \alpha)),
\end{equation*}
there exists $\bar \xi\in [\xi_0, \xi_2)$ unique such that $t(\bar \xi, \eta_0)= \frac12 \alpha+ \varepsilon$ and 
\begin{equation*}
\mu\vert _{\tilde J_t}= \mu_{ac}\vert_{\tilde J_t} \quad \text{ for } \quad \frac12 \alpha\leq t\leq  \frac12 \alpha+ \varepsilon
\end{equation*}
where $\tilde J_t=[y(t, \bar \xi), y(t, \xi_4)] \supset [y(0, \bar \xi)-t, \alpha(1- t)]\supset[y(0, \bar \xi)-t,0]$, since $c(x)\geq 1$. Finally, let $y(0, \bar \xi)= \delta$ and hence $J_t= [-\delta t, 0]$ to finish the proof of \eqref{ref:claim}.

\begin{remark}
In \cite{BCZ} the uniqueness of conservative solutions to the nonlinear variational wave equation has been established. One of the properties characterising this class of solutions is the following:
For  almost every $t\in \Real$, the singular parts of $\mu(t)$ and $\nu(t)$ are concentrated on the set where $c'(u)=0$. 

As the above example illustrates the set where $c'(u)=0$ plays a key role, since it allows for conservative solutions such that there exists a set $\mathcal{B}\subset\Real$ with $\mathrm{meas}(\mathcal{B})>0$ such that for every $t\in \mathcal{B}$ either the measure $\mu(t)$ or $\nu(t)$ is not purely absolutely continuous. The main underlying reason is the fact that in regions where $c'(u)=0$ the solution behaves like those for the linear wave equation and hence concentrated energy cannot be spread out. 

For the conservative solutions of the Camassa--Holm and the Hunter--Saxton equation, on the other hand, the situation is quite different. One has to require that there exists a null set $\mathcal{N}\subset \Real$ such that for every $t\not \in \mathcal{N}$ the measure $\mu$ is absolutely continuous. Since $u$ for those equations plays the same role as $c(u)$ for the NVW equation, the imposed criteria is in alignment with the one in \cite{BCZ} if $c'(x)\not =0$ for all $x\in \Real$. 
\end{remark}
  
  \appendix
  
 \section{Useful estimates}

\begin{proposition}\label{prop:ber}
Assume $h(t)<0$ and $ \gamma(t)\leq 0$ for all $t\in [t_0,t_1]$ and that 
\begin{equation}\label{intineq2}
  a + \int_{t_0}^{t}  \gamma h^2(s) ds\leq h(t)\leq \tilde a + \int_{t_0}^{t}  \gamma h^2(s) ds<0, 
\end{equation}
then 
\begin{equation*}
  \frac{a}{1- a\int_{t_0}^t  \gamma (s) ds}\leq h(t)\leq \frac{ \tilde a}{1-  \tilde a\int_{t_0}^t  \gamma (s) ds}.
\end{equation*}
\end{proposition}

\begin{proof} We only establish the upper bound, since the lower bound can be established following the same lines.

Consider the function 
\begin{equation*}
  z(t)= \tilde a+ \int_{t_0}^t  \gamma h^2(s) ds,
\end{equation*}
which satisfies $h(t)\leq z(t)<0$ for all $t\in [t_0, t_1]$ by \eqref{intineq2}, $z(t_0)=\tilde a<0$, and 
\begin{equation*}
  z'(t)=  \gamma h^2(t)\leq  \gamma z^2(t).
\end{equation*}
Since the above differential inequality can be solved explicitly, we end up with 
\begin{equation*}
  h(t)\leq z(t)\leq \frac{\tilde a}{1-\tilde a \int_{t_0}^t \gamma(s) ds}.
\end{equation*}
\end{proof}

\end{document}